\documentclass[11pt,reqno]{amsart}
\usepackage[a4paper,left=3.5cm,right=3.5cm,top=3cm,bottom=3cm]{geometry}
\usepackage{amsfonts}
\usepackage{amscd}
\usepackage{amssymb}
\usepackage{amsmath} 
\usepackage{hyperref}
\usepackage{mathrsfs}
\usepackage{dsfont}
\usepackage{enumerate}
\usepackage{xcolor}
\usepackage{float} 
\usepackage[pdftex]{graphicx}
\allowdisplaybreaks
\usepackage{mathtools}
\mathtoolsset{showonlyrefs}
\usepackage[most]{tcolorbox}
\usepackage[font=small,labelfont=bf]{caption}
\usepackage{soul}
\usepackage{multicol}

\theoremstyle{plain}
\newtheorem{theorem}{Theorem}[section]

\newtheorem{corollary}[theorem]{Corollary}
\newtheorem{definition}[theorem]{Definition}
\newtheorem{lemma}[theorem]{Lemma}
\newtheorem*{problem*}{Problem}

\newtheorem{proposition}[theorem]{Proposition}

\newtheorem*{theorem*}{Theorem}
\usepackage{mathtools}
\usepackage{amssymb}
\usepackage{graphicx}

\newcommand{\dotLongleftarrow}{%
  \mathrel{%
    \tikz[baseline=-0.55ex]{
      \draw[line width=0.35pt,
            dotted,
            line cap=round]
        (0.20,0.05) -- (0.65,0.05);
      \draw[line width=0.35pt,
            dotted,
            line cap=round]
        (0.20,-0.05) -- (0.65,-0.05);

      \draw[line width=0.35pt,
            line cap=round,
            line join=round]
        (0.22,0.11) -- (0.1,0) -- (0.22,-0.11);
    }%
  }%
}
\newcommand{\dotlongleftarrow}{%
  \mathrel{%
    \tikz[baseline=-0.55ex]{
      \draw[line width=0.35pt,
            dotted,
            line cap=round]
        (1.50,0.05) -- (2.15,0.05);
      \draw[line width=0.35pt,
            dotted,
            line cap=round]
        (1.50,-0.05) -- (2.15,-0.05);

      \draw[line width=0.35pt,
            line cap=round,
            line join=round]
        (1.52,0.11) -- (1.4,0) -- (1.52,-0.11);
    }%
  }%
}
\newcommand{\dotUparrow}{\rotatebox{-90}{$\dotlongleftarrow$}}
\newcommand{\dotDownarrow}{\rotatebox{90}{$\dotlongleftarrow$}}

\theoremstyle{definition}
\newtheorem{example}[theorem]{Example}

\newtheorem{remark}[theorem]{Remark}

\newtheorem{question}{Question}

\newcommand{\rr}[1]{\mathbb{R}^{#1}}
\newcommand{\inn}{\textnormal{in}\ }
\newcommand{\onn}{\textnormal{on}\ }
\newcommand{\norm}[2]{\left|\left| #1 \right| \right|_{#2}}

\newcommand{\pder}[2]{\frac{\partial #1}{\partial #2}}
\newcommand{\dint}{\textnormal{d}}
\newcommand{\gbox}[1]{\begin{tcolorbox}[breakable, enhanced]#1 \end{tcolorbox}}

\newcommand{\andd}{\textnormal{and}}

\newcommand{\dvol}{\ \textnormal{dv}}

\usepackage[english]{babel}

\usepackage{amsmath}
\usepackage{graphicx}

\title[]{$L^p$-Liouville properties and realizations of Schr\"odinger operators on incomplete Riemannian manifolds}
\author{Andrea Bisterzo, Salvatore Lincastri and Giona Veronelli}

\address{Andrea Bisterzo
  \newline \indent Centro di Ricerca Matematica Ennio De Giorgi
  \newline \indent
  Scuola Normale Superiore di Pisa
\newline\indent Piazza dei Cavalieri 3, 56126, Pisa, Italy.}
\email{andrea.bisterzo@sns.it}

\address{Salvatore Lincastri
  \newline \indent Dipartimento di Matematica e Applicazioni
  \newline \indent
  Università degli Studi di Milano-Bicocca
\newline\indent Via Roberto Cozzi 55, 20125, Milano, Italy.}
\email{s.lincastri@campus.unimib.it}

\address{Giona Veronelli
  \newline \indent Dipartimento di Matematica e Applicazioni
  \newline \indent
  Università degli Studi di Milano-Bicocca
\newline\indent Via Roberto Cozzi 55, 20125, Milano, Italy.}
\email{giona.veronelli@unimib.it}

\begin{document}

\begin{abstract}
The aim of the paper is to highlight and standardize in a self-contained framework the connections between some properties associated with the Laplace-Beltrami operator on possibly incomplete Riemannian manifolds. We present both new and well-known characterizations of the positivity-preserving property, essential self-adjointness, and Liouville-type results for $L^2$-functions, using the latter to establish a clear hierarchy among these properties.

Furthermore, we prove the stability of the essential self-adjointness of the Laplace-Beltrami operator on incomplete Riemannian manifolds under the removal of closed (possibly non-compact) subsets from a complete manifold, subject to suitable conditions on their (generalized) Hausdorff measure.

With appropriate modifications, this framework is adapted throughout the paper to the general $L^p$-case for $p\in (1,+\infty)$.
\end{abstract}
\maketitle


\medskip 


\section{Introduction}

Liouville-type theorems are among the most basic manifestations of rigidity in elliptic partial differential equations. In their classical form, they assert that a solution of a homogeneous elliptic equation satisfying a suitable global growth or integrability condition must be trivial. For the Laplace operator on Euclidean space, the underlying mechanism is remarkably simple. If \(u\geq 0\) is subharmonic on \(\mathbb{R}^n\) and belongs to \(L^p(\mathbb{R}^n)\), \(1\leq p<+\infty\), the mean value inequality and Hölder's inequality give
\begin{align}
    u(x)
 \leq \frac{1}{|B_r(x)|}\int_{B_r(x)}u
 \leq |B_r|^{-1/p}\|u\|_{L^p(\mathbb{R}^n)},
\end{align}
and letting \(r\to+\infty\) yields \(u(x)\leq0\), and thus \(u\equiv0\). In particular, every \(L^p\)-harmonic function on \(\mathbb{R}^n\) vanishes. Thus, in the Euclidean setting, global \(L^p\)-integrability by itself provides a very strong rigidity condition. The same result, i.e. the constancy of nonnegative subharmonic functions, holds on more general $geodesically$ $complete$ Riemannian manifolds: indeed, S.T. Yau proved (\cite{yau1976some}) the $L^{p}$-Liouville property for non-negative, smooth subharmonic functions for every $p\in(1,+\infty)$ (see also \cite{LiSchoen84}, where the property for $p\in(0,1]$ is also considered under additional geometric assumptions on the manifold, \cite{PigolaRS-Vanishing} for analogous results in the weighted setting and the more recent \cite{PVV} for the case of distributional subharmonic functions).

The situation becomes substantially subtler as soon as the underlying space is no longer complete: removing even a very small set may create singular harmonic or subharmonic functions that are globally integrable. The fundamental solution of the Laplacian already provides elementary examples: for instance, suitable truncations of the Green function yield nonconstant, nonnegative \(L^2\)-subharmonic functions on punctured Euclidean spaces, while combinations of Green functions with opposite poles can produce nontrivial \(L^2\)-harmonic functions. This naturally leads to the following question.
\begin{question}
Given a possibly incomplete Riemannian manifold \((M,g)\), which geometric or analytic properties prevent the existence of nontrivial \(L^p\)-harmonic or nonnegative \(L^p\)-subharmonic distributions?
\end{question}

One of our main motivations for studying this question is the connection between \(L^p\)-Liouville phenomena and uniqueness properties for realizations of the Laplacian and Schr\"odinger operators, a theory deeply rooted in non-relativistic quantum mechanics and in von Neumann's mathematical formulation of quantum mechanics; we refer to \cite[pp. 642-643]{BMS} for a nice introduction to the physical origin of the problem.

\medskip

Here and throughout the paper, unless otherwise specified, \(M\) is a smooth, connected Riemannian manifold without boundary, not necessarily geodesically complete. The Liouville properties are formulated in terms of the Laplace--Beltrami operator
$
\Delta=\operatorname{div}\nabla\cdot
$
with the sign convention for which \(-\Delta\) is nonnegative. For \(p\in(1,+\infty)\), we shall mainly consider the following two Liouville properties:
\begin{align}
    (L^p\text{-Liou}_H)&: & u\in L^p(M),\quad \Delta u=0
\quad\Longrightarrow\quad
u\ \text{is constant},\\
(L^p\text{-Liou}_{SH})&: &0\leq u\in L^p(M),\quad \Delta u\geq0
\quad\Longrightarrow\quad
u\ \text{is constant}.
\end{align}
We shall also consider their shifted counterparts
\((L^p\text{-Liou}_{\lambda-H})\) and
\((L^p\text{-Liou}_{\lambda-SH})\), obtained by replacing \(\Delta\) with \(\Delta-\lambda\), where \(\lambda\) is constant.

It is not difficult to show that these Liouville properties form a hierarchy of progressively weaker rigidity statements: for any $\lambda>0$
\begin{align}
(L^p\text{-Liou}_{SH})
\Longrightarrow
(L^p\text{-Liou}_{\lambda-SH})
\Longrightarrow
(L^p\text{-Liou}_{\lambda-H})
\Longrightarrow
(L^p\text{-Liou}_H).
\end{align}
\noindent It is particularly interesting to understand how the operator-theoretic properties alluded to above fit into the above chain of implications. Let us first restrict to the functionally simpler Hilbertian case \(p=2\). A standard criterion for essential self-adjointness shows that the essential self-adjointness of
\begin{equation*}
    -\Delta\big|_{C^\infty_c(M)}
\end{equation*}
in \(L^2(M)\), which we shall denote by \((ESA)\), is equivalent to the property
\begin{equation*}
    (L^2\text{-Liou}_{\lambda-H})
\end{equation*}
for any \(\lambda>0\). Essential self-adjointness of the Laplacian on (possibly incomplete) manifolds has by now been extensively studied; see, for instance, \cite{strichartz1983analysis} for the complete case, \cite{BGS} for the complete case with boundary, and \cite{DeVerd,CdVT10,BMS,hinz2023removable,masammune1999essential, PRS18} and the references therein for the incomplete case.

For the Laplacian, one does not need the full strength of the distributional \(L^2_{\lambda-SH}\)-Liouville property in order to obtain essential self-adjointness: it is enough to rule out suitable $L^2$-subsolutions enjoying additional Sobolev regularity. The situation changes when one considers a semibounded Schr\"odinger operator
\begin{align}
-\Delta+V,
\qquad
0\leq V\in L^2_{\mathrm{loc}}(M).
\end{align}
In this case the natural deficiency equation only gives, a priori, distributional solutions, and the positive and negative parts have to be treated at this low level of regularity. Consequently, the full \(L^2_{\lambda-SH}\)-Liouville property for distributional \(L^2\)-subsolutions becomes relevant.

This approach goes back, in different forms, to Kato and was later developed in the geometric setting in \cite{kato1972schrodinger,BMS}. It is usually formulated through a closely related property, the so-called \(L^2\)-positivity preservation property:
\begin{align}
    (L^2\text{-PP}):
\qquad
u\in L^2(M),\quad (-\Delta+1)u\geq0
\quad\Longrightarrow\quad
u\geq0.
\end{align}
A simple but useful observation is that
\begin{align}
    (L^2\text{-PP})
\quad\Longrightarrow\quad
(L^2\text{-Liou}_{1-SH}).
\end{align}
Indeed, if \(0\leq u\in L^2(M)\) satisfies \(\Delta u\geq u\), then
\begin{align}
    (-\Delta+1)(-u)\geq0,
\end{align}
and positivity preservation gives \(-u\geq0\), hence \(u\equiv0\).
Conversely, the distributional Brezis--Kato inequality shows that
\begin{align}
    (L^2\text{-Liou}_{1-SH})
\quad\Longrightarrow\quad
(L^2\text{-PP}),
\end{align}
so that the two properties are in fact equivalent. 

The question of establishing positivity preservation on complete manifolds turned out to be surprisingly nontrivial. In particular, Braverman, Milatovic and Shubin conjectured that every geodesically complete Riemannian manifold should satisfy the \(L^2\)-positivity preservation property. This problem, usually referred to as the BMS conjecture, motivated a substantial amount of work and was only recently settled; see \cite{BMS,BS18,guneysu2016sequences, guneysu2017covariant, gu-BMS,MV24,PVV}. One of the key difficulties is that the relevant subsolutions are merely distributional. The solution relies on a regularity theory for nonnegative subharmonic distributions, based in particular on smooth monotone approximation and local Sobolev estimates; see \cite{PVV} (moreover, we refer to \cite{GPSV24} for an extension to non-smooth Dirichlet spaces and \cite{bisterzo2024loc,alias2024growth} for the positivity preservation for locally integrable distributions satisfying suitable growth conditions).

A further observation is that, in the $L^2$ setting, positivity preservation is actually equivalent to the unshifted Liouville property for nonnegative subharmonic distributions:
\begin{align}
(L^2\text{-PP})
\quad\Longleftrightarrow\quad
(L^2\text{-Liou}_{SH}).
\end{align}
Thus, in the Hilbertian setting, positivity preservation and the \(L^2\)-Liouville property for nonnegative subharmonic distributions are two equivalent formulations of the same rigidity phenomenon. We can summarize the relations alluded to so far in the following scheme
\
\begin{center}

\begin{theorem}\label{thm_main_L^2} Let $(M,g)$ be a (possibly incomplete) Riemannian manifold. Then, the following implications hold\\ \ 

\scalebox{0.8}{\begin{minipage}{15cm}
\begin{align}
\left.\begin{array}{ccccccc}
&&&  &  \begin{array}{c}
\text{ESA}_\lambda\ \text{for every}  \\ \lambda \in L^\infty(M)\end{array}\\
&&&& \left. \text{\footnotesize{\quad \ }} \, \Big\Updownarrow \, \text{\footnotesize{(3)}}\right. \\
\left. \begin{array}{c}L^2-PP_\lambda \quad \text{for some} \\ \text{(equiv. any)}\ \lambda>0 \end{array}\right. & \left.\begin{array}{c}{\footnotesize \text{(1)}} \\ \Longrightarrow\end{array}\right. & \left. \begin{array}{c} \text{ESA}_V\ \text{for every}\\ V\in L^2_{loc}(M), V\geq C\end{array}\right. & \left.\begin{array}{c}{\footnotesize\text{(2)}} \\ \Longrightarrow\end{array}\right.& \left. \text{ESA}\right.   \\
\left. \text{\ \ \footnotesize{(4)}} \, \Big\Updownarrow \, \text{\footnotesize{\quad \quad \quad}} \right. && \vcenter{\offinterlineskip\vskip30pt}  &&\left. \text{\footnotesize{\quad \ }} \, \Big\Updownarrow \, \text{\footnotesize{(5)}} \right. \\
\left. \begin{array}{c}L^2-\text{Liou}_{\lambda-SH} \quad \text{for some} \\ \text{(equiv. any)}\ \lambda>0 \end{array}\right. &&  \left. \begin{array}{c}
    \overset{\text{\footnotesize{(6)}}}{\Longrightarrow} \\ \underset{\textcolor{black}{\text{\footnotesize{(7)}}}}{\textcolor{black}{\dotLongleftarrow}}
\end{array}\right.& & \left. \begin{array}{c}L^2-\text{Liou}_{\lambda-H} \quad \text{for some} \\ \text{(equiv. any)}\ \lambda>0 \end{array}\right. \\
\left. \text{\footnotesize{(8)}}\, \Big\Updownarrow  \, \text{\footnotesize{\quad \quad}} \right. &&&&\quad \quad \ \left. \begin{array}{cc}  \begin{minipage}{2cm}\text{\footnotesize{(9)}}\, \Big\Downarrow \end{minipage}& \begin{minipage}{2cm} \hspace{-1cm} \vspace{-0.1cm}\dotUparrow\, \end{minipage} \hspace{-2.5cm}\text{\footnotesize{(10)}}\end{array}\right. \\
L^2-Liou_{SH} && \vcenter{\offinterlineskip\vskip30pt} && L^2-Liou_{H}
\end{array}\right. 
\end{align}
\end{minipage}}
\hspace{-12.2cm}
\begin{minipage}{4cm}
\begin{tcolorbox}[breakable, enhanced, colback=white, opacityback = 0, opacityframe = .1] \ \\ \ \\ \ \\ \ \\ \ \\ \ \\ \ \\ \ \\ \ \\ \ \end{tcolorbox}\vspace{-0.8cm}
\end{minipage}
\hspace{5.5cm}
\begin{minipage}{4cm}
\begin{tcolorbox}[breakable, enhanced, colback=white, opacityback = 0, opacityframe = .1] \ \\ \ \\ \ \\ \ \\ \ \\ \ \\ \ \\ \ \\\ \end{tcolorbox}  \vspace{1.8cm}
\end{minipage}\\
where
\begin{enumerate}
    \item follows from \cite{BMS} (see Proposition \ref{prop_L2pp_ESA-D+V});

    \item follows from \cite{BMS} (see Corollary \ref{Cor_ESA_V_ESA_lambda});

    \item follows from combining some arguments contained in \cite{RS2, MT16, BV24} (see also Corollary \ref{cor_equiv_ESA});

    \item follows from Corollary \ref{cor_equivalence_L2PP_Liouville};

    \item follows from classical arguments, see \cite{RS2} (see also Proposition \ref{prop_L_l-H_ESA_-D+l});

    \item follows from Lemma \ref{lem_l-SH_implies_l-H};

    \item follows from Proposition \ref{Prop_positive cone}, under the additional assumption that $C^\infty_c(M)_+$ is dense in $H^2_+$;

    \item follows from Corollary \ref{cor_equivalence_L2SH_Liouville_L2PP};

    \item follows from combining some arguments contained in \cite{RS2} together with \cite[Corollary 4]{HMW} (see also Proposition \ref{prop_L_l-H_ESA_-D+l} and Proposition \ref{ESA-->L2H});

    \item follows from \cite[Corollary 4]{HMW}, under the assumptions that $-\Delta$ has positive spectrum and $\textnormal{vol}(M)=+\infty$ (see also Proposition \ref{ESA-->L2H}).
\end{enumerate}
\end{theorem}
\end{center}
\medskip

In this manuscript, the goal is to determine conditions ensuring the essential self-adjointness (and related spectral properties) of quite general Schr\"odinger operators, requiring only minimal integrability assumptions and a constant lower bound on the potential. Other works in the literature have also investigated the interplay between the Cauchy boundary of an incomplete manifold and the behavior of the potential near it. Namely, an improved lower bound on $V$ allows for weaker assumptions on the singularity of $M$; see, e.g., \cite{MT16,PRS18,BV24}.
\medskip

When \(p\neq2\), the previous picture becomes more subtle. Essential self-adjointness is intrinsically a Hilbert space notion, and in \(L^p\) there are at least two natural properties playing an analogous role. The first states that \(C^\infty_c(M)\) is an operator core for the maximal realization of \(-\Delta+V\) in \(L^p(M)\), i.e. that the minimal realization coincides with the maximal realization; see Section \ref{Sec_basics} for the rigorous definitions. For \(V=0\), this is equivalent to the density of \(C^\infty_c(M)\) in
\begin{align}
H^p(M):=\{u\in L^p(M) : \Delta u\in L^p(M)\}
\end{align}
with respect to the graph norm
\begin{align}
\|u\|_{H^p}=\|u\|_{L^p}+\|\Delta u\|_{L^p}.
\end{align}
The second concerns semigroups: one may ask whether the minimal realization generates a contraction semigroup on \(L^p(M)\). Contrary to the Hilbertian case, these two properties do not coincide in general, the semigroup-generation property being weaker than the core property.

A further difficulty comes from duality. If \(p'=p/(p-1)\), the adjoint relation between minimal and maximal realizations naturally exchanges the exponents \(p\) and \(p'\). Consequently, an \(L^p\)-positivity preserving property does not directly yield a statement for the minimal realization on \(L^p\), but rather implies that
\begin{align}
(-\Delta+V)_{p',\min}
\end{align}
generates a contraction semigroup on \(L^{p'}(M)\). To recover the stronger core property on \(L^p\), one is naturally led to require positivity preservation at both conjugate exponents \(p\) and \(p'\). Under suitable assumptions on \(V\), the simultaneous \(L^p\)- and \(L^{p'}\)-positivity preservation properties imply that \(C^\infty_c(M)\) is a core for \((-\Delta+V)_{p,\max}\).

Thus, for \(p\neq2\), the theory necessarily involves an interplay between two different Banach spaces, \(L^p\) and \(L^{p'}\). When \(p=2\), this distinction disappears since \(p=p'\), and the two operator-theoretic viewpoints collapse into the familiar essential self-adjointness framework.
\medskip

\begin{theorem}\label{thm_main_L^p}
    Let $(M,g)$ be a (possibly incomplete) Riemannian manifold. Fix any $p\in (1,+\infty)$ and denote by $p'=\frac{p}{p-1}$ its conjugate exponent. Then, the following implications hold\\
\hspace{-1cm}  \scalebox{0.7}{\begin{minipage}{15cm}
\begin{align}
\left.\begin{array}{ccccc}
\left. \begin{array}{c}L^p-PP_\lambda \quad \text{for some} \\ \text{(equiv. any)}\ \lambda>0 \end{array}\right. & \left.\begin{array}{c}{\footnotesize \text{(1)}} \\ \Longrightarrow\end{array}\right. & \left. \begin{array}{c} (-\Delta+V)_{p',min}\ \text{generates a}\\ \text{contraction semigroup in } L^{p'}(M)\\ \text{for every } V\in L^{p'}_{loc}(M), V\geq 0\end{array}\right. & \left.\begin{array}{c}{\footnotesize\text{(2)}} \\ \Longrightarrow\end{array}\right.&  \left. \begin{array}{c} (-\Delta)_{p',min}\ \text{generates}\\ \text{a contraction semigroup}\\ \text{in } L^{p'}(M)\end{array}\right.   \\
\left. \text{\ \ \footnotesize{(3)}} \, \Big\Updownarrow \, \text{\footnotesize{\quad \quad \quad}} \right. && \vcenter{\offinterlineskip\vskip30pt}  &&\left. \text{\footnotesize{\quad \ }} \, \Big\Updownarrow \, \text{\footnotesize{(4)}} \right. \\
\left. \begin{array}{c}L^p-\text{Liou}_{\lambda-SH} \quad \text{for some} \\ \text{(equiv. any)}\ \lambda>0 \end{array}\right. &&  
    \overset{\text{\footnotesize{(5)}}}{\Longrightarrow}  && \left. \begin{array}{c}L^p-\text{Liou}_{\lambda-H} \quad \text{for some} \\ \text{(equiv. any)}\ \lambda>0 \end{array}\right. \vspace{0.2cm} \\
\left. \text{\footnotesize{(3)}}\, \Big\Updownarrow  \, \text{\footnotesize{\quad \quad}} \right. &&&&\quad \quad \ \left. \begin{array}{cc}  \begin{minipage}{2cm}\text{\footnotesize{(6)}}\, \Big\Downarrow \end{minipage}& \begin{minipage}{2cm} \hspace{-1cm} \vspace{-0.1cm}\dotUparrow\, \end{minipage} \hspace{-2.5cm}\text{\footnotesize{(7)}}\end{array}\right.\\
L^p-Liou_{SH} && \vcenter{\offinterlineskip\vskip30pt} && L^p-Liou_{H}\\
\left. \begin{minipage}{1cm}
\vspace{-0.5cm}    \text{\footnotesize{(8)}}
\end{minipage}\,\hspace{-0.4cm} \dotDownarrow  \, \text{\footnotesize{\quad \quad}} \right. \vspace{0.3cm}&&&&\\
\left. \begin{array}{c}\text{For any fixed}\\ V\in L^p_{loc}(M)\cap L^{p'}_{loc}(M),\\ \ V\geq 0:\end{array}\right. &&&&\\ 
\left.\begin{array}{c}C^\infty_c(M) \text{ is a core for}\\ (-\Delta+V)_{p,max}\end{array}\right. & \left.\begin{array}{c}{\footnotesize \text{(9)}} \\ \Longleftrightarrow\end{array}\right. &\left.\begin{array}{c}C^\infty_c(M) \text{ is a core for}\\ (-\Delta+V)_{p',max}\end{array}\right. & \left.\begin{array}{c}{\footnotesize \text{(9)}} \\ \Longleftrightarrow\end{array}\right. & \left. \begin{array}{c} (-\Delta+V)_{p,min} \text{ generates a}\\ \text{contraction semigroup in } L^p(M)\\ \text{and } (-\Delta+V)_{p',min} \text{ generates a}\\ \text{contraction semigroup in } L^{p'}(M)\end{array}\right.
\end{array}\right.
\end{align}
\end{minipage}}
\hspace{-10.6cm}
\begin{minipage}{3.5cm}\vspace{0.8cm}
\begin{tcolorbox}[breakable, enhanced, colback=white, opacityback = 0, opacityframe = .1]\vspace{0.3cm}\ \\ \ \\ \ \\ \ \\ \ \\ \  \\
\end{tcolorbox}\vspace{4cm}
\end{minipage}
\hspace{7cm}
\begin{minipage}{3.3cm}\vspace{0.6cm}
\begin{tcolorbox}[breakable, enhanced, colback=white, opacityback = 0, opacityframe = .1] \vspace{0.2cm}\ \\ \ \\ \ \\ \ \\  \end{tcolorbox} \vspace{3.4cm} \ \\ \ \\ \ \\
\end{minipage}

\hspace{-0.25cm}
\begin{minipage}{14.2cm}\vspace{-4cm}
\begin{tcolorbox}[breakable, enhanced, colback=white, opacityback = 0, opacityframe = .1]\vspace{0.7cm}\ \\ \ \\ \ \\ \  \end{tcolorbox}
\end{minipage}\\
where
\begin{enumerate}
    \item follows from combining some arguments contained in \cite{RS2, milatovic2010m, BV24} (see Proposition \ref{prop_LpPP_CSG-D+V});
    \item follows as the particular case $V\equiv 0$;
    \item follows from Corollary \ref{cor_equivalence_L2SH_Liouville_L2PP};
    \item follows from combining some arguments contained in \cite{RS2, milatovic2010m, brezis2011functional} (see Proposition \ref{Lp lambda armomiche e semigruppo});
    \item follows from Lemma \ref{lem_l-SH_implies_l-H};
    \item follows by Theorem \ref{Thm_L^p_lambda-H implies L^p_H};
    \item follows from Proposition \ref{Lp lambda armomiche e semigruppo}, under the assumption that $-\Delta$ has positive spectrum on $L^2$ and $\textnormal{vol}(M)=+\infty$;
    \item follows from Corollary \ref{prop_LpPP+Lp'PP_Core}, under the additional assumption that the $L^{p'}$-PP holds;
    \item follows from combining some arguments contained in \cite{RS2, Ka13, milatovic2010m} (see Proposition \ref{prop_semigroups equivalent cores}).
\end{enumerate}
\end{theorem}

\begin{remark}
    We want to remark that in both Theorems \ref{thm_main_L^2} and \ref{thm_main_L^p} the range of admissible constants $\lambda$ can be enlarged to include (certain) negative values. We refer to the relative sections for further details.
\end{remark}

We stress that among the new results proved in the present paper, and used mainly as technical tools for proving some implications contained in Theorems \ref{thm_main_L^2} and \ref{thm_main_L^p}, some are particularly noteworthy and of independent interest. Among them, we highlight the following proposition, proved in Subsection \ref{sec_monotonicity}, concerning the monotonicity of the $L^p$-Liouville property for subharmonic functions (and, hence, of all properties equivalent to it) as the integration exponent $p$ increases.
\begin{proposition}
    Let $(M,g)$ be a (possibly incomplete) Riemannian manifold and fix $\lambda\geq 0$. If $M$ satisfies the $(L^p\textnormal{-Liou}_{\lambda-SH})$ property for $p\in [1,+\infty)$, then it satisfies the $(L^q\textnormal{-Liou}_{\lambda-SH})$ for every $q\in[p,+\infty)$.
\end{proposition}

Following the (almost) complete pictures presented in Theorem \ref{thm_main_L^2} and \ref{thm_main_L^p}, we investigate some sufficient conditions for some of the above properties in the wide class of incomplete Riemannian manifolds obtained by removing a closed (possibly non-compact) set from a complete manifold. As is well known in the literature (see \cite{hinz2023removable, PVV, BV24}), some conditions on the Hausdorff dimension have to be imposed. For fixed $p\in (1,+\infty)$, denoting by $\mathfrak{M}_p$ a generalized $(n-2p')$-dimensional Hausdorff measure (see Section \ref{Sec_removed sets}), the content of Theorems \ref{Liouville subharm on manif with cls rem} and \ref{Schrod on manif with cls rem} is the following

\begin{theorem}\label{thm_main_remove_closed_set}
    Let $(N,h)$ be a complete Riemannian manifold of dimension $n\ge3$, $p\in \left[\frac{n}{n-2},+\infty\right)$ and $K\subseteq N$ be a closed set. Consider the open manifold $M=N\setminus K$ and assume that   
    \begin{equation*}
        \mathfrak{M}_p(\widetilde{K})<+\infty \quad\quad \text{for every compact $\widetilde{K}\subseteq K$}.
    \end{equation*}
    Then, ($L^{q}\textnormal{-Liou}_{SH}$) holds on $M$ for every $q\in[p,+\infty)$.
\end{theorem}

\begin{theorem}\label{thm_main_remove_closed_set2}
    Let $(N,h)$ be a complete Riemannian manifold of dimension $n\ge3$, $p\in \left(1,\frac{n}{2}\right]$ and $K\subseteq N$ a closed set. Consider the open manifold $M:=N\setminus K$ and let $0\le V\in L^{p}_{loc}(M)$. If
    \begin{equation*}
        \mathfrak{M}_{p'}(\widetilde{K})<+\infty \quad\quad \text{for every compact $\widetilde{K}\subseteq K$},
    \end{equation*}
    then $(-\Delta+V)_{q,min}$ generates a contraction semigroup in $L^{q}(M)$ for all $1<q\le p$.
\end{theorem}

The previous result extends \cite[Theorem 10.12.(ii)]{hinz2023removable} to the case of Schrodinger operators, and improves \cite[Proposition 7.1]{PVV}, generalizing the results to the case of incomplete Riemannian manifolds obtained by removing closed, possibly non compact, subsets, and giving better information when $n=2p'$. Moreover, we would like to point out that the statement of Proposition 7.1(a) in \cite{PVV} requires a minor correction. Indeed, as Example \ref{contract but not core} shows, under the assumption that $\textnormal{dim}_{\mathcal{H}}(K)<n-2p'$ one can only ensure that $(-\Delta+V)_{q,min}$ generates a contraction semigroup in $L^q(M)$ for every $1<q\leq p'$, without reaching the case $q=p$ if $p>2$.

It is worth noticing that the assumptions on $K$ in Theorems \ref{thm_main_remove_closed_set} and \ref{thm_main_remove_closed_set2} are essentially sharp, in the sense that as soon as the Hausdorff dimension of $K$ exceeds the indicated threshold, there are counterexamples; see for instance Theorem \ref{thm:negPVV} and \cite{hinz2023removable}. 

A crucial role in the proof of the Theorems \ref{thm_main_remove_closed_set} and \ref{thm_main_remove_closed_set2} is played by the following proposition on the removability of singularities for distributional subharmonic functions. We refer to Proposition \ref{Removability of singularity} for the proof.
\begin{proposition}
    Let $(M,g)$ be a (possibly incomplete) Riemannian manifold of dimension $n\ge3$, $p\in \left[\frac{n}{n-2},+\infty\right)$ and $K\subseteq M$ be a closed set. If
    \begin{equation*}
        \mathfrak{M}_{p}(\widetilde{K})<+\infty \quad\quad \text{for every compact $\widetilde{K}\subseteq K$},
    \end{equation*}
    then
    \begin{equation*}
        \begin{cases}
            \Delta u\ge0 &\text{in $\mathcal{D}'(M\setminus K)$}\\
            0\le u\in L^{p}_{loc}(M)
        \end{cases}
        \Longrightarrow \Delta u\ge0 \quad\text{in $\mathcal{D}'(M)$.}
    \end{equation*}
\end{proposition}

The paper is organized as follows. In Section \ref{Sec_basics}, we introduce the $L^p$-Liouville properties for (non-negative $\lambda$-sub)harmonic functions, and we recall some basic notions from potential theory and operator theory, providing the necessary background for the development of the paper. In Section \ref{Sec_LpPP_ESA_Core_CSG}, we investigate the connection between the $L^p$-positivity preserving property and certain functional analytic properties associated with Schrodinger operators: namely, we study the contraction property for the semigroup generated by the $p$-minimal extension of the operator and the density of $C^{\infty}_{c}(M)$ in the domain of the $p$-maximal realization of the operator. When $p=2$, these two properties are equivalent to the operator being essentially self-adjoint. In Section \ref{Sec_class_Lp_Liou_Monot}, we study the relationships between the various $L^p$-Liouville properties and the $L^p$-positivity preservation property, establishing a precise hierarchy among them and clarifying their connections with the functional analytic properties introduced in the previous sections. A monotonicity property for the $L^p$-Liouville property for non-negative subharmonic functions with respect to $p\in[1,+\infty)$  is also discussed. Finally, in Section \ref{Sec_removed sets} we investigate the validity of the $L^p$-Liouville property for non negative subharmonic functions on incomplete Riemannian manifolds obtained by removing a closed, not necessarily compact, subset from a complete manifold. We show that the property remains valid when the removed set satisfies suitable conditions on its (generalized) Hausdorff measure, and discuss the consequences for realizations of Schrodinger operators with locally integrable potential bounded from below.


\section{Basics on potential and operator theory}\label{Sec_basics}
Let $(M, g)$ be a connected, smooth, possibly incomplete Riemannian manifold without boundary, equipped with the Riemannian volume measure $\textnormal{dv}$ and the associated non-positive definite Laplace-Beltrami operator $\Delta$.  Unless otherwise specified, we’ll always assume this setting.

\subsection{Potential theory}\ \\
The scope of the present subsection is to introduce the notions of Liouville-type properties and positivity preservation. To this aim, we start recalling
some basic facts about distributional (sub/super)solutions to elliptic partial differential equations.

\subsubsection{Preliminaries on distributional solutions} Let us fix $V$ and $f$ in $L^1_{loc}(M)$.

\begin{definition}[Distributional (sub/super)solution]
    We say that $u\in L^{1}_{loc}(M)$ is a \textnormal{distributional supersolution} to $(-\Delta+V)u= f$ (or a distributional solution to  $(-\Delta+V)u\ge f$)  if $Vu\in L^1_{loc}(M)$ and
    \begin{equation}\label{distrib. supers}
        \int_{M}u(-\Delta+V)\varphi \dvol \ge \int_{M}f\varphi \dvol\quad \quad \forall \varphi\in C^{\infty}_{c}(M),\ \varphi\ge0.
    \end{equation}
    Analogously, we say that $u$ is a \textnormal{distributional subsolution} to $(-\Delta+V)u= f$ (or a distributional solution to $(-\Delta+V)u\le f$) if $Vu\in L^1_{loc}(M)$ and $u$ satisfies \eqref{distrib. supers} with the reversed inequality. Lastly, $u$ is said to be a \textnormal{distributional solution} to $(-\Delta+V)u=f$ if $u$ is both a supersolution and a subsolution.
\end{definition}

In the special case $f\equiv 0$, we say that $u\in L^1_{loc}(M)$ is a $V$\textit{-superharmonic distribution} (resp. $V$\textit{-subharmonic  distribution} or $V$\textit{-harmonic distribution}) if it is a distributional supersolution (resp. subsolution or solution) to $(-\Delta+V)u= 0$.

If the function $u$ is assumed to be more regular, for instance $u\in W^{1,1}_{loc}(M)$, we say that $u$ is a \textit{weak supersolution} to $(-\Delta+V)u= f\in L^{1}_{loc}(M)$ if 
    \begin{equation}\label{weak sol}
        \int_{M}\left[g(\nabla u,\nabla\varphi)+V u\varphi\right]\dvol\ge\int_{M}f\varphi \dvol \quad \quad \forall\varphi\in C^{\infty}_{c}(M),\ \varphi\ge 0.
    \end{equation}
Similarly, $u$ is a \textit{weak subsolution} to the same equation if it satisfies \eqref{weak sol} with the reversed inequality. Finally, it is a weak solution if it is both a subsolution and a supersolution. 

\begin{remark}
    Note that, by a density argument, the non-negative test functions in \eqref{weak sol} can also be taken in $W^{1,\infty}_{c}(M)$. Moreover, if $V\in L_{loc}^\infty$ and $u\in W^{1,2}_{loc}(M)$ then the test functions can be taken in $W^{1,2}_{c}(M)\cap L^\infty(M)$.
\end{remark}

Under suitable assumptions on the potential $V$, several tools can be employed to address the lack of regularity of the distributional solutions, which are initially assumed to belong to $L^{1}_{loc}$. A first key result is a local monotone approximation theorem through smooth $V$-subharmonic functions.

\begin{theorem}[\cite{bisterzo2023positivity, PVV}]\label{monotone approx thm}
    Let $(M,g)$ be a (possibly incomplete) Riemannian manifold, $0\le V\in C^{\infty}(M)$ and let $u\in L^{1}_{loc}(M)$ be a $V$-subharmonic distribution. Then, for every precompact open set $\Omega\subseteq M$ with smooth boundary there exists a sequence $\{u_{k}\}_{k\ge1}\subseteq C^{\infty}(\overline{\Omega})$ satisfying the following properties:
    \begin{enumerate}
        \item $u\le u_{k+1}\le u_{k}$ for every $k\ge1$;
        \item $u_{k}\to u$ almost everywhere in $\Omega$ as $k\to +\infty$;
        \item $u_{k}$ is a $V$-subharmonic function for every $k\ge1$;
        \item $u_k \to u$ in $L^1(\Omega)$  as $k\to +\infty$.
    \end{enumerate}
    Moreover, if $u\geq 0$, $u\not \equiv 0$, then the sequence $u_{k}$ can be chosen to be strictly positive. 
\end{theorem}

\begin{remark}
    The fact that if $u\ge0$ then the sequence $u_{k}$ can be chosen to be strictly positive is not explicitly claimed in  \cite{bisterzo2023positivity} and \cite{PVV} but can be deduced from the proof of the theorem.
\end{remark}

\noindent The second result states that we can improve the $L^{1}_{loc}$-regularity of a subharmonic distribution.

\begin{proposition}[Theorem 3.1 in \cite{PVV}]\label{Reg_sub_dist}
    Let $(M,g)$ be a (possibly incomplete) Riemannian manifold and let $0\le u\in L^{1}_{loc}(M)$ be a subharmonic distribution. Then $u^{\frac{s}{2}}\in L^{\infty}_{loc}(M)\cap W^{1,2}_{loc}(M)$ for every $s\in(1,+\infty)$. 
\end{proposition}
\noindent Finally, we recall the Brezis-Kato inequality, which is contained in the following
\begin{proposition}[Theorem 4.1 in \cite{PVV}]\label{Thm_Brezis-Kato}
    Let $(M,g)$ be a (possibly incomplete) Riemannian manifold. If $f\in L^{1}_{loc}(M)$ and $u\in L^{1}_{loc}(M)$ satisfies $\Delta u\ge f$ in the sense of distributions, then $u_{+}$ is a distributional solution to $\Delta u_{+}\ge1_{u>0}f$. 
\end{proposition}
\noindent In particular, the positive part of a $V$-subharmonic distribution is still a $V$-subharmonic distribution.

\subsubsection{Liouville properties and positivity preservation} Let us fix $\lambda\in \rr{}$.

\begin{definition}[Liouville property]
     Given $\mathcal{S}\subseteq L^1_{loc}(M)$, we say that the $\mathcal{S}$-\textnormal{Liouville property for $\lambda$-subharmonic functions} holds if
    \begin{align}
        \begin{cases}
            (-\Delta+\lambda) u\leq 0 & in\ \mathcal{D}'(M)\\
            u\geq 0 & a.e.\ in\ M\\
            u\in \mathcal{S}
        \end{cases} \quad \Rightarrow \quad u\equiv const.\ in\ M.
    \end{align}
    Similarly, we say that the $\mathcal{S}$-\textnormal{Liouville property for $\lambda$-harmonic functions} holds if
        \begin{align}
        \begin{cases}
            (-\Delta+\lambda) u= 0 & in\ \mathcal{D}'(M)\\
            u\in \mathcal{S} 
        \end{cases} \quad \Rightarrow \quad u\equiv const.\ in\ M.
    \end{align}
\end{definition}

\noindent In what follows, given $\mathcal{S}\subseteq L^1_{loc}(M)$ and $\lambda \in \rr{}$, we use
\begin{itemize}
    \item ($\mathcal{S}$-Liou$_{SH}$) to mean the $\mathcal{S}$-Liouville property for subharmonic functions;
    \item ($\mathcal{S}$-Liou$_{\lambda-SH}$) to mean the $\mathcal{S}$-Liouville property for $\lambda$-subharmonic functions;
    \item ($\mathcal{S}$-Liou$_{H}$) to mean the $\mathcal{S}$-Liouville property for harmonic functions;
    \item ($\mathcal{S}$-Liou$_{\lambda - H}$) to mean the $\mathcal{S}$-Liouville property for $\lambda$-harmonic functions.
\end{itemize}
\medskip

A second property we are going to study is the so-called \textit{positivity preservation}; see \cite[Section XIV.5]{guneysu2017covariant}.

\begin{definition}[Positivity preservation]
   Given $\mathcal{S}\subseteq L^1_{loc}(M)$, we say that the $\mathcal{S}$-\textnormal{positivity preservation for $-\Delta+\lambda$} holds if
    \begin{align}
        \begin{cases}
            (-\Delta+\lambda)u\geq 0 & in\ \mathcal{D}'(M)\\ u\in \mathcal{S}
        \end{cases}\quad \Rightarrow \quad u\geq 0\ a.e.
    \end{align}
\end{definition}

\noindent Given $\mathcal{S}\subseteq L^1_{loc}(M)$ and $\lambda \in \rr{}$, we use
\begin{itemize}
    \item $(\mathcal{S}$-$\textnormal{PP}_\lambda$) to mean the $\mathcal{S}$-positivity preservation for $-\Delta+\lambda$;
    \item $(\mathcal{S}$-PP) to mean $(\mathcal{S}$-$\textnormal{PP}_1$).
\end{itemize}

\begin{remark}\label{rmk_resolvent_L^2-PP}
    We observe that the terminology \textit{positivity preservation} comes from the fact that, in the case $\mathcal{S}=H^2(M):=\{u\in L^2(M)\ :\ \Delta u \in L^2(M)\}$, the ($H^2(M)$-PP) (which implies the essential self-adjointness of the Laplacian, and hence the fact that $(-\Delta, H^2(M))$ coincides with the Friedrichs Laplacian) is equivalent to the property of the resolvent $R:L^2(M)\to H^2(M)$ of $\mathcal{Q}:=(-\Delta+1)$ to preserve positivity (i.e. the order). Indeed, assume that $\mathcal{Q}$ is $H^2(M)$-positivity preserving and fix $0\leq f\in L^2(M)$. Then, defining $u=Rf$ we have
    \begin{align}
        \mathcal{Q} u=\mathcal{Q}Rf=f\geq 0,
    \end{align}
    implying $Rf=u\geq 0$ by the $H^2(M)$-positivity preservation of $\mathcal{Q}$. Hence $R$ preserves the positivity. Conversely, if $R$ preserves the positivity, then for any fixed $u\in H^2(M)$ so that $\mathcal{Q}u\geq 0$ one has $u=R\mathcal{Q}u\geq 0$. As a consequence, if the $H^2(M)$-positivity preservation for $\mathcal{Q}$ holds, then for every $f,h\in L^2(M)$
    \begin{align}
        f\geq h \quad \Rightarrow \quad Rf\geq Rh.
    \end{align}
\end{remark}

\subsection{Operator theory}\ \\
We proceed by collecting some essential concepts and results from operator theory that will prove useful in exploring the $L^2$-Liouville properties on general Riemannian manifolds.

\subsubsection{Basics on semigroups} We begin with the following definition.

\begin{definition}[Strongly continuous and contraction semigroup]
    A family of bounded operators $\{P_t\}_{t\geq 0}$, where $P_t:L^p(M)\to L^p(M)$ for $p\in [1,+\infty)$, is a \textnormal{strongly continuous semigroup} if
    \begin{itemize}
        \item $P_0=Id$;
        \item $P_sP_t=P_{s+t}$ for all $s,t\geq 0$;
        \item $\forall u \in L^p(M)$ the map $t\mapsto P_tu$ is continuous in $L^p(M)$.
    \end{itemize}
    A strongly continuous semigroup $\{P_t\}_{t\ge0}$ is said to be a \textnormal{contraction semigroup} if
    \begin{align}
        \norm{P_t}{\mathcal{L}(L^p)}\leq 1 \quad \forall t\geq 0,
    \end{align}
    where $\norm{\cdot}{\mathcal{L}(L^p)}$ denotes the operator norm.
\end{definition}

\noindent One of the main properties of strongly continuous semigroups is that they are generated by closed operators (in the following sense).

\begin{proposition}[P. 237 in \cite{RS2}]
    Let $\{P_t\}_{t\ge0}$ be a strongly continuous semigroup acting on $L^p(M)$ and set $T_t:=t^{-1} (Id-P_t)$ and $T:=\lim_{t\to 0} T_t$, defined over
    \begin{align}
        D(T):=\{u\in L^p(M) : \lim_{t\to 0}T_t u\ \text{exists in } L^p(M)\}.
    \end{align}
    Then, $T$ is closed and densely defined. The operator $T$ is called the infinitesimal generator of $P_t$ and we say that $T$ generates $P_t$.
\end{proposition}

\begin{remark} We recall that an operator $(T,D(T))$ defined on $L^p(M)$ is said to be \textit{closed} if its graph
\begin{align}
    \Gamma(T):=\{(u,Tu)\ : u\in D(T)\}
\end{align}
is closed in $L^p(M)\times L^p(M)$ (or, equivalently, if $(\Gamma(T),\norm{\cdot}{T})$ is complete, where $\norm{u}{T}:=\norm{u}{L^p}+\norm{Tu}{L^p}$ for $u\in \mathcal{D}(T)$). An operator is \textit{closable} if it possesses a closed extension.
\end{remark}

The following theorem provides a necessary and sufficient condition for closed operators to generate a contraction semigroup (see \cite[Theorem X.47a]{RS2}).

\begin{theorem}[Hille-Yosida]\label{Hille_Yosida}
    A densely defined, closed linear operator $T$ on $L^p(M)$ generates a contraction semigroup if and only if both the following conditions are satisfied
    \begin{enumerate}
        \item $(-\infty,0)\subset \rho(T)$, where $\rho(T)$ denotes the resolvent set of $T$;
        \item $\norm{(T+\lambda)^{-1}}{\mathcal{L}(L^P)}\leq \lambda^{-1}$ for all $\lambda>0$.
    \end{enumerate}
\end{theorem}

Another characterization of generators of contraction semigroups is stated in the following part of the section. To this aim, we start by recalling some definitions.

\begin{definition}
    Given $u\in L^p(M)$, $p \in (1,+\infty)$, a function $w\in L^{p'}(M)$ ($p':=\frac{p}{p-1}$ is the conjugate exponent to $p$) is said to be a \textnormal{normalized tangent functional} to $u$ if it satisfies
    \begin{align}
        \norm{w}{L^{p'}}=\norm{u}{L^p} \quad and \quad w(u)=\norm{u}{L^p}^2
    \end{align}
    where
    \begin{align}
        w(u):=\int_M uw \dvol.
    \end{align}
    A densely defined operator $T$ over $L^p(M)$ is said to be \textnormal{accretive} if for any $u\in D(T)$ there exists a normalized tangent functional $w\in L^{p'}(M)$ to $u$ so that $w(Tu)\geq 0$. An accretive operator $T$ is said to be \textnormal{maximal accretive} ($m$-accretive) if it has no proper accretive extensions.
\end{definition}

\begin{remark}
    We observe that if $u\in L^p(M)$, then there always exists a normalized tangent functional to $u$. Indeed, the function $J_p^u:=\|u\|_{L^p}^{2-p} |u|^{p-2}u$ (with $J_p^0:=0$) is an element of $L^{p'}(M)$ and satisfies
    \begin{align}
        \norm{J_p^u}{L^{p'}}=\norm{u}{L^p} \quad \andd \quad J_p^u(u)=\norm{u}{L^p}^2.
    \end{align}
    In particular, by the remark before Theorem X.49 in \cite{RS2}, if $\mathrm{Ran}(T+1)=L^p(M)$, then $T$ is accretive if and only if for any $u\in D(T)$ it holds $J_p^u (Tu)\geq 0$.
\end{remark}

\begin{theorem}[Theorem X.48 in \cite{RS2}]\label{t:X.48}
    A closed operator $T$ on $L^p(M)$ is the generator of a contraction semigroup if and only if $T$ is accretive and $\mathrm{Ran}(\bar{\lambda}+T)=L^p(M)$ for some $\bar{\lambda}>0$.
\end{theorem}

\begin{remark} We recall that
    \begin{enumerate}
        \item every accretive operator is closable and its closure (i.e. the minimal closed extension)  is accretive (\cite[Section X.8]{RS2});

        \item every accretive operator has a smallest closed accretive extension (\cite[Section X.48]{RS2});

        \item by Theorem \ref{Hille_Yosida}, if $T$ is the generator of a contraction semigroup, then $(-\infty, 0)\subset \rho(T)$ and it follows immediately that $\mathrm{Ran}(1+T)=L^p(M)$;

        \item the generators of contraction semigroups are maximal accretive since the condition $\mathrm{Ran}(1+T)=L^p(M)$ implies that $T$ has no proper accretive extensions. The converse ($T$ maximal accretive implies $T$ generates a contraction semigroup) holds if $p=2$ (Hilbert space) but not in general $L^p$. See \cite[Page 241]{RS2}.	
    \end{enumerate}
\end{remark}

\subsubsection{Essential self-adjointness and operator cores} Let $T:D(T)\subset L^2(M)\to L^2(M)$ be a densely defined linear operator with domain $D(T)$.

\begin{definition}[Adjoint, symmetric and (essentially) self-adjoint operator]
    The \textnormal{adjoint} of T, denoted with $T^*$, is defined as the linear operator acting on $L^2(M)$ whose domain is
    \begin{equation*}
        D(T^*):=\{v\in L^2(M) : \exists w \in L^2(M)\ s.t.\ (Tu,v)_{L^2(M)} =(u,w)_{L^2(M)} \quad \forall u \in D(T)\}
    \end{equation*}
    and whose action is given by
    \begin{align}
        T^*v=w.
    \end{align}
    Hence
    \begin{align}
        (Tu,v)=(u,T^*v) \quad\quad \forall u\in D(T),\ v \in D(T^*).
    \end{align}
    The operator $T$ is said to be \textnormal{symmetric} if $D(T)\subseteq D(T^*)$ and
    \begin{align}
        (Tu,v)=(u,Tv)\quad \forall u,v \in D(T).
    \end{align}
    A symmetric operator $T$ is said to be  \textnormal{self-adjoint} if $T=T^*$ (i.e., if $D(T)=D(T^*)$). A symmetric operator $T$ is said to be \textnormal{essentially self-adjoint} if its closure $\overline{T}$ is self-adjoint.
\end{definition}

\begin{remark}
    We stress that
    \begin{enumerate}
        \item the adjoint of an operator is a closed operator. In particular, if $T$ is symmetric (resp. self-adjoint), then $T$ is closable (resp. closed);

        \item it is well known that if $T$ is closable, then $(T^*)^*=\overline{T}$ (see \cite[Theorem 5.29]{Ka13});

        \item a symmetric operator $T$ is essentially self-adjoint if and only if it has a unique self-adjoint extension, i.e. its closure (see \cite[p. 256]{RS1}).
    \end{enumerate}
\end{remark}

\noindent In what follows we use $(ESA_{V})$ to mean the essential self-adjointness of the operator $-\Delta + V$ with domain $C^\infty_c(M)$, where $V\in L^2_{loc}(M)$ satisfies suitable assumptions.

\begin{remark}
    For fixed $p\in(1,+\infty)$, we denote the \textit{duality product} between $L^p(M)$ and its topological dual space $(L^{p}(M))^*$ as the bilinear form
    \begin{align}
        <\cdot ,\cdot> :(L^{p}(M))^*\times L^{p}(M) &\to \rr{}\\
        (\varphi,u)&\mapsto \varphi(u).
    \end{align}
    By the standard canonical isomorphism between $(L^{p}(M))'$ and $L^{p'}(M)$, one has that every $\varphi\in (L^{p}(M))'$ can be identified with a (unique) $v_\varphi \in L^{p'}(M)$, and vice versa. Under this identification, it holds
    \begin{align}
        <\varphi,u>=\int_M uv_\varphi \dvol \quad \forall \varphi \in (L^{p}(M))',\ \forall u\in L^p(M).
    \end{align}
    For a fixed densely defined linear operator $T:D(T)\subset L^p(M)\to L^p(M)$, let us denote by $T'$ its dual map, i.e. the linear operator $T':D(T')\subset (L^{p}(M))'\to (L^p(M))'$ with domain
    \begin{align}
        D(T'):=\{\varphi \in (L^p(M))' \mid \ \exists c>0\ :\ \vert<\varphi,T(u)>\vert\leq c \norm{u}{L^p} \ \forall u \in D(T)\}
    \end{align}
    and acting on $\varphi\in D(T')$ as
    \begin{align}
        T'(\varphi)(u)=\varphi(T(u)) \quad \forall u\in D(T),
    \end{align}
    i.e.
    \begin{align}
        <T'(\varphi),u>=<\varphi,T(u)> \quad \forall u\in D(T).
    \end{align}
    Observing that
    \begin{align}
        <\varphi,u>=(v_\varphi,u)_{L^2(M)} \quad \forall \varphi \in (L^p(M))', \forall u \in L^p(M)
    \end{align}
    and that
    \begin{align}
        (v_\varphi, T(u))_{L^2(M)}&=<\varphi, T(u)>\\
        &=<T'(\varphi),u>\\
        &=(v_{T'(\varphi)},u)_{L^2(M)},
    \end{align}
    we have that the following linear operator is well defined
    \begin{align}
        T^*:D(T^*)&\to L^{p'}(M)\\
        v_\varphi &\mapsto v_{T'(\varphi)}
    \end{align}
    with domain
    \begin{align}
        D(T^*):=\{v_\varphi \in L^{p'}(M) : \varphi \in D(T')\}
    \end{align}
    and it satisfies
    \begin{align}
        (T^*(v),u)_{L^2(M)}=(v, T(u))_{L^2(M)} \quad \forall v \in D(T^*), \quad \forall u\in D(T).
    \end{align}
We denote $T^*$ with the same notation of the adjoint of an operator defined on a Hilbert space, and we call $T^*$ the \textit{adjoint} to $T$ even if $p\neq 2$, due to the similarity with the notion of adjoint in the case $p=2$.  
\end{remark}

For general $p$, a possible definition which can be seen as a Banach generalization of the $(ESA)$ is the following notion of \textit{core} of a closed operator. Let us fix $p\in (1,+\infty)$.

\begin{definition}[Core in $L^p$]
    Given a densely defined, closed operator 
    \begin{equation*}
        T:D(T)\subseteq L^p(M)\to L^p(M),
    \end{equation*} 
    a subset $\mathcal{C}\subset D(T)$ is said to be a \textnormal{core} for $(T,D(T))$ if the set $\{(v,Tv) : v\in \mathcal{C}\}$ is dense in the graph $\Gamma(T)$ of $T$.
\end{definition}

\begin{remark}
    As already observed, in the case $p=2$, the essential self-adjointness can also be formulated in terms of the notion of core of an operator. Indeed, we can say that a symmetric closable operator $T$ on $L^2$ is essentially self-adjoint if and only if its domain is a core for $T^*$.
\end{remark} 

Let $V\in L^1_{loc}(M)$ and consider the differential operator $-\Delta+V$. If $p\in(1,+\infty)$, we define the operator $(-\Delta+V)_{p,max}$ associated to $-\Delta+V$ by the formula
\begin{align*}
(-\Delta+V)_{p,max} u = (-\Delta+V) u
\end{align*}
with domain
\begin{align}
D\left((-\Delta+V)_{p,max}\right)=\{u\in L^p(M)\ : Vu\in L^1_{loc}(M),\ (-\Delta+V)u\in L^p(M)\}.
\end{align}
If $V\in L^p_{loc}(M)$, then $C^\infty_c(M)\subset D((-\Delta+V)_{p,max}))$ and we can define the operator $(-\Delta+V)_{p,min}$ as the closure in $L^p(M)$ of $(-\Delta+V)_{p,max}\Big|_{C_c^\infty(M)}$.

\begin{remark}\label{rmk_ESA}
Let $V\in L^2_{loc}(M)$ and consider $P:=-\Delta+V$ with domain $D(P):=C^\infty_c(M)\subset L^2(M)$. Then, the domain of the adjoint $P^*$ of $P$ can be characterized as
\begin{align}
    D(P^*):=\{v\in L^2(M) \mid Pv\in L^2(M)\}.
\end{align}
Hence, $P^*=\left(\overline{P}\right)^*$ coincides with the maximal extension of $P$ (see for instance \cite[Theorem X.32]{RS2}). 
\end{remark}

In the particular case $V=0$, the fact that $C^\infty_c(M)$ is a core for  $-\Delta_{p,max}$ can be read as the identification of the Sobolev spaces $H^{p}_{0}(M)$ and $H^{p}(M)$, where
\begin{align}
    &H^p_0(M):=D(-\Delta_{p,min})=\overline{C^\infty_c(M)}^{\norm{\cdot}{-\Delta_{p,max}}},\\
    &H^p(M):=D(-\Delta_{p,max})=\{u\in L^p(M) : \Delta u \in L^p(M)\},
\end{align}
for $\norm{u}{-\Delta_{p,max}}$ the graph-norm associated to $-\Delta_{p,max}$. For a simplified notation, we denote
\begin{align}
    \norm{\cdot }{H^p}:=\norm{\cdot}{-\Delta_{p,max}}.
\end{align}
Namely, the following Lemma holds.
\begin{lemma}\label{H2,0 = H2}
    Let $(M,g)$ be a (possibly incomplete) Riemannian manifold. Then, $C^\infty_c(M)$ is a core for $-\Delta_{p,max}$ if and only if $H^p(M)=H^p_0(M)$.
\end{lemma}

\begin{remark}\label{weak convergence of H2=H20}
    Since every function $\varphi\in H^2_0(M)$ can be approximated by a sequence $\{\varphi_{k}\}_{k\ge1}\subset C^{\infty}_{c}(M)$ in the norm $\norm{\cdot}{H^2}$, it follows that on $H^2_0(M)$ the norm $\norm{\cdot}{H^2}$ is equivalent to
    \begin{align}
        |||\varphi|||:=\norm{\varphi}{L^2}+\norm{\nabla \varphi}{L^2}+\norm{\Delta \varphi}{L^2}.
    \end{align}
    Indeed, fixed any $\varphi\in H^2_0(M)$ and $\{\varphi_{k}\}_k\subset C^\infty_c(M)$ so that $\varphi_k\xrightarrow[]{\norm{\cdot}{H^2}}\varphi$, integrating by parts we get
    \begin{align}      \label{App0_norm_equivalence}
    \int_{M}\vert\nabla(\varphi_{k}-\varphi_{l})\vert^{2} \dvol &\le-\int_{M}(\varphi_{k}-\varphi_{l})\Delta(\varphi_{k}-\varphi_{l}) \dvol \\
    &\le\int_{M}(\varphi_{k}-\varphi_{l})^{2} \dvol +\int_{M}(\Delta(\varphi_{k}-\varphi_{l}))^{2} \dvol
    \end{align}
    and hence $\{\varphi_k\}_k$ is a Cauchy-sequence in $W^{1,2}(M)$. By the completeness of $W^{1,2}(M)$, there exists $\nabla \varphi \in L^2(M)$ so that $\nabla\varphi_k\to \nabla \varphi$ in $L^2(M)$. 
    
    As a consequence, if $\varphi_k$ converges to $\varphi$ in $H^2_0(M)$, then $\nabla\varphi_k$ converges strongly to $\nabla \varphi$ in $L^2$. In particular, if the Laplace-Beltrami operator (with domain $C^\infty_c(M)$) is essentially self-adjoint on $M$, then the same conclusion extends to $H^2(M)=H^2_0(M)$ thanks to Lemma \ref{H2,0 = H2}.
\end{remark}


Here we collect some conditions concerning the core property. The first follows by an argument presented in \cite{BV24} and claims that the core property (and thus the essential self-adjointness) is invariant by translation, in the sense that if $V\in L^{p}_{loc}(M)$ and $C^{\infty}_{c}(M)$ is a core for $(-\Delta+V)_{p,max}$, then $C^{\infty}_{c}(M)$ is a core for $(-\Delta+V+\lambda)_{p,max}$ for every $\lambda\in L^\infty(M)$. 
\begin{proposition}\label{prop_invariance_translation}
    Let $(M,g)$ be a (possibly incomplete) Riemannian manifold, $p\in (1,+\infty)$ and $V\in L^{p}_{loc}(M)$. The following statements are equivalent: 
    \begin{enumerate}
        \item $C^{\infty}_{c}(M)$ is a core for $(-\Delta+V)_{p,max}$; 
        \item $C^{\infty}_{c}(M)$ is a core for $(-\Delta+V+\lambda)_{p,max}$ for every $\lambda\in L^{\infty}(M)$.
    \end{enumerate} 
\end{proposition}
\begin{proof}
    Assume that $C^{\infty}_{c}(M)$ is a core for $(-\Delta+V)_{p,max}$ and fix $\lambda\in L^{\infty}(M)$. Note that 
    \begin{equation*}
        D((-\Delta+V+\lambda)_{p,max}) = D((-\Delta+V)_{p,max}),
    \end{equation*}
    and that for every $u\in D((-\Delta+V)_{p,max})$ there exists a sequence $\{u_n\}_{n\ge1}\subseteq C^{\infty}_{c}(M)$ such that 
    \begin{equation*}
        u_{n}\xrightarrow{L^{p}} u \quad \text{and}\quad (-\Delta+V)u_{n}\xrightarrow{L^{p}}(-\Delta+V)u.
    \end{equation*}
    Having fixed $u\in D((-\Delta+V+\lambda)_{p,max})=D((-\Delta+V)_{p,max})$ and a sequence $\{u_n\}_{n\ge1}\subseteq C^{\infty}_{c}(M)$ as before, using Minkowski's inequality we get that
    \begin{align}
        \Vert (-\Delta+V&+\lambda)(u_{n}-u) \Vert_{L^{p}(M)} \\
        &\le \Vert (-\Delta+V)(u_{n}-u) \Vert_{L^{p}(M)}+\Vert\lambda\Vert_{L^\infty(M)} \Vert u_{n}-u\Vert_{L^{p}(M)}\longrightarrow 0,
    \end{align}
    so $u_n$ converges to $u$ in the graph norm of $(-\Delta+V+\lambda)_{p,max}$. This implies that $C^{\infty}_{c}(M)$ is a core for $(-\Delta+V+\lambda)_{p,max}$, and the proof is concluded.
\end{proof}

\begin{remark}
    We stress that the previous proposition can be slightly generalized by replacing $C^\infty_c(M)$ with a general core $\mathcal{C}$ for $(-\Delta+V)_{p,max}$.
\end{remark}

\begin{remark}
    In the case $p=2$ and $\lambda\in\mathbb{R}$ is constant, the previous Proposition is a particular case of Proposition 4.1 in \cite{MT16}. 
\end{remark}

\noindent When $p=2$ and $V\equiv0$, the result reads as follows.
\begin{corollary}\label{cor_equiv_ESA}
    Let $(M,g)$ be a (possibly incomplete) Riemannian manifold. The following statements are equivalent:
    \begin{enumerate}
        \item $(M,g)$ is $(ESA)$;
        \item $(M,g)$ is $(ESA_{\lambda})$ for \underline{every} $\lambda\in L^\infty(M)$.
    \end{enumerate}
\end{corollary}

\noindent The last result we recall in this subsection is very classical in the literature. 

\begin{theorem}[Theorem X.26 in \cite{RS2}]\label{thm_RS2_X26}
    Let $A$ be a strictly positive symmetric operator, i.e. $(A\varphi,\varphi)\geq c(\varphi,\varphi)$ for all $\varphi \in D(A)$ and some $c>0$. Then the following are equivalent:
    \begin{enumerate}
        \item $A$ is essentially self-adjoint;
        \item $\mathrm{Ran}(A)$ is dense;
        \item $\mathrm{ker}(A^*)=\{0\}$;
        \item $A$ has only one semibounded self-adjoint extension.
    \end{enumerate}
\end{theorem}

\subsubsection{$L^{p}$-bounded solutions to $-\Delta u + \lambda u = f$}
The present subsection aims to prove a result which will be useful in the next sections. The proof builds on techniques developed in \cite{BMS,MV24}, which are suitably refined to account for the lower regularity of the function involved and for the bottom of the spectrum of the Laplacian.

\begin{proposition}\label{solution with non-neg Lp function}
    Let $(M,g)$ be a (possibly incomplete) Riemannian manifold and let $p\in(1,+\infty)$. For every $0\le f\in L^{p}(M)$ and $\lambda>0$ there exists a nonnegative  solution $v\in H^{p}(M)$ to the equation
    \begin{equation*}
        (-\Delta + \lambda)v=f \quad\text{on $M$}
    \end{equation*}
    such that
    \begin{equation*}
        \Vert v\Vert_{L^{p}(M)}\le \frac{1}{C_{p,\lambda}}\Vert f\Vert_{L^{p}(M)},
    \end{equation*}
    where $C_{p,\lambda}:=\frac{4}{pp'}\lambda_{1}(M)+\lambda$ and $\lambda_{1}(M)\ge0$ is the bottom of the spectrum of $-\Delta$.
\end{proposition}

\begin{proof} 
Let $0\le f\in L^{p}(M)$ and assume $f\not\equiv 0$ (otherwise the claim trivially follows). Pick an increasing sequence $\{\Omega_{k}\}_{k\ge1}$ of precompact open sets of $M$ with smooth boundaries such that $\Omega_k\Subset\Omega_{k+1}$, $M=\bigcup_{k\ge1}\Omega_k$ and $\Omega_{k}\cap\{f>0\}\ne\emptyset$, and consider a sequence of smooth cutoff functions $\{\chi_{k}\}_{k\ge1}\subset C^{\infty}(M)$ such that 
    \begin{equation*}
        \begin{cases}
            0\le\chi_{k}\le1& \\
            \chi_{k}\equiv 1 & \text{on $\Omega_{k-1}$} \\
            \chi_{k}\equiv 0& \text{outside a neighborhood of $\overline{\Omega}_{k-1}$ contained in $\Omega_{k}$.}
        \end{cases}
    \end{equation*}
    For every $n\ge1$, consider the non-decreasing sequence $f_n :=\min(n,f)\chi_{n}$, and for every $k\ge n$ let $v^{(n)}_{k}$ be the solution of the Dirichlet problem
    \begin{equation}\label{system Lp}
        \begin{cases}
            (-\Delta+\lambda)v^{(n)}_{k}=f_{n} &\text{on $\Omega_{k}$}\\
            v^{(n)}_{k}=0 &\text{on $\partial\Omega_{k}$.}
        \end{cases}
    \end{equation}
    Since $0\le f_{n}\in L^{\infty}(M)$ and has compact support, then by elliptic regularity we have $v^{(n)}_{k}\in W^{2,q}_{loc}(\Omega_{k})\cap C(\overline{\Omega_{k}})$ for every $q>1$, and so $v^{(n)}_{k}\in C^{1,\alpha}(\overline{\Omega_{k}})$ by Morrey embedding (see Chapter 7 and 9 of \cite{gilbarg1998elliptic}). Moreover, every $v^{(n)}_{k}$ is $\lambda$-superharmonic, and by the maximum principle for strong solutions we have
    \begin{equation*}
        0< v^{(n)}_{k}(x)< v^{(n)}_{k+1}(x) \quad \forall k\ge n, \forall x\in \Omega_{k}.
    \end{equation*}    
    By monotonicity, we have the pointwise limit
    \begin{equation*}
        v^{(n)}(x):=\lim_{k\rightarrow+\infty}v^{(n)}_{k}(x)\in(0,+\infty] \quad \forall x\in M.
    \end{equation*}
    We want to prove that the function $v^{(n)}$ is finite almost everywhere and in $L^{p}(M)$. To do this, multiply the equation in \eqref{system Lp} by $(v_{k}^{(n)})^{p-1}$ and integrate over $\Omega_{k}$ to obtain \begin{equation*}
        \begin{split}
            \int_{\Omega_{k}}(v_{k}^{(n)})^{p-1}(-\Delta+\lambda)v^{(n)}_{k} &= (p-1)\int_{\Omega_{k}}(v_{k}^{(n)})^{p-2}\vert\nabla v_{k}^{(n)}\vert^{2}+\lambda\int_{\Omega_{k}}(v_{k}^{(n)})^{p}\\
            & =\frac{4(p-1)}{p^{2}}\int_{\Omega_{k}}\vert\nabla (v_{k}^{(n)})^{\frac{p}{2}}\vert^{2}+\lambda\int_{\Omega_{k}}(v_{k}^{(n)})^{p}
        \end{split}
    \end{equation*}
    and
    \begin{equation*}
        \begin{split}
            \int_{\Omega_{k}}(v_{k}^{(n)})^{p-1}(-\Delta+\lambda)v^{(n)}_{k} &= \int_{\Omega_{k}}(v_{k}^{(n)})^{p-1}f_{n} \\
            &\le\bigg(\int_{\Omega_{k}}(v_{k}^{(n)})^{p}\bigg)^{\frac{p-1}{p}} \bigg(\int_{\Omega_{k}}f_{n}^{p}\bigg)^{\frac{1}{p}}\\
            &\le\bigg(\int_{\Omega_{k}}(v_{k}^{(n)})^{p}\bigg)^{\frac{p-1}{p}} \bigg(\int_{M}f^{p}\bigg)^{\frac{1}{p}},
        \end{split}
    \end{equation*}
    so that
    \begin{equation}\label{stima prima di autovalore_approssimanti a livello n}
        \frac{4(p-1)}{p^{2}}\int_{\Omega_{k}}\vert\nabla (v_{k}^{(n)})^{\frac{p}{2}}\vert^{2}+\lambda\int_{\Omega_{k}}(v_{k}^{(n)})^{p} \le\bigg(\int_{\Omega_{k}}(v_{k}^{(n)})^{p}\bigg)^{\frac{p-1}{p}} \bigg(\int_{M}f^{p}\bigg)^{\frac{1}{p}}.
    \end{equation}
    In particular, since $v^{(n)}_{k}\in C^{1,\alpha}(\overline{\Omega_{k}})$ is positive in $\Omega_{k}$ and vanishes on $\partial\Omega_{k}$, the last estimate implies that the function $(v^{(n)}_{k})^{\frac{p}{2}}$ is in $W^{1,2}(\Omega_{k})$ and vanishes on $\partial\Omega_{k}$. From the variational characterization of the first Dirichlet eigenvalue and the monotonicity with respect to the inclusion of domains, it follows that 
    \begin{equation}\label{maggiorazione dal basso per v_k^n con primo autovalore}
        \int_{\Omega_{k}}\vert\nabla (v_{k}^{(n)})^{\frac{p}{2}}\vert^{2}\ge \lambda_{1}(\Omega_{k})\int_{\Omega_{k}}(v_{k}^{(n)})^{p} \ge \lambda_{1}(M)\int_{\Omega_{k}}(v_{k}^{(n)})^{p},
    \end{equation}
    and combining \eqref{stima prima di autovalore_approssimanti a livello n} and \eqref{maggiorazione dal basso per v_k^n con primo autovalore} we obtain
    \begin{equation*}
        \bigg(\frac{4(p-1)}{p^{2}}\lambda_{1}(M)+\lambda\bigg)\int_{\Omega_{k}}(v_{k}^{(n)})^{p} \le\bigg(\int_{\Omega_{k}}(v_{k}^{(n)})^{p}\bigg)^{\frac{p-1}{p}} \bigg(\int_{M}f^{p}\bigg)^{\frac{1}{p}},
    \end{equation*}
    so that
    \begin{equation*}
        \bigg(\int_{\Omega_{k}}(v_{k}^{(n)})^{p}\bigg)^{\frac{1}{p}} \le \frac{1}{C_{p,\lambda}} \bigg(\int_{M}f^{p}\bigg)^{\frac{1}{p}},
    \end{equation*}
    where $C_{p,\lambda}:=\frac{4(p-1)}{p^{2}}\lambda_{1}(M)+\lambda=\frac{4}{pp'}\lambda_{1}(M)+\lambda$.
    From the monotone convergence theorem and the definition of $v^{(n)}$ we deduce that $0\le v^{(n)}\in L^{p}(M)$ and 
    \begin{equation*}
        \Vert v^{(n)}\Vert_{L^{p}(M)}\le \frac{1}{C_{p,\lambda}}\Vert f\Vert_{L^{p}(M)}.
    \end{equation*}
    Moreover, testing the equation of $v_{k}^{(n)}$ with a general $0\le\varphi\in C^{\infty}_{c}(M)$ and using the dominated convergence theorem we also get that $v^{(n)}$ solves
    \begin{equation*}
        -\Delta v^{(n)}+\lambda v^{(n)} = f_{n}
    \end{equation*}
    on $M$ in the sense of distribution. This implies that $\Delta v^{(n)} = \lambda v^{(n)}-f_{n}\in L^{p}(M)$, hence $v^{(n)}\in H^{p}(M)$.
    
    Finally, we have to prove that $v^{(n)}$ converges in $H^{p}(M)$ to a solution of the equation
    \begin{equation*}
        -\Delta v +\lambda v = f.
    \end{equation*}
    To see this, note that the same computation used to obtain the bound on the $L^{p}$ norm of $v^{(n)}_{k}$ can be performed to obtain the same bound on the function $w_{n,m}:=v^{(n)}-v^{(m)}$, $n\ge m\ge1$, which is a nonnegative (by the maximum principle) $H^{p}$ solution to
    \begin{equation*}
        -\Delta w_{n,m}+\lambda w_{n,m} = f_{n}-f_{m}\ge0
    \end{equation*}
    on $M$. In particular, by using an exhaustion argument with the relatively compact open sets $\Omega_k$ and the functions $v^{(n)}_k$ defined at the beginning of the proof, we get the estimate \begin{equation*}
        \Vert v^{(n)}-v^{(m)}\Vert_{L^{p}(M)} = \Vert w_{n,m}\Vert_{L^{p}(M)}\le \frac{1}{C_{p,\lambda}}\Vert f_{n}-f_{m}\Vert_{L^{p}(M)},
    \end{equation*}
    which also implies
    \begin{equation*}
        \Vert \Delta v^{(n)}-\Delta v^{(m)}\Vert_{L^{p}(M)} = \Vert \Delta w_{n,m}\Vert_{L^{p}(M)}\le \bigg(\frac{\lambda}{C_{p,\lambda}}+1\bigg)\Vert f_{n}-f_{m}\Vert_{L^{p}(M)}.
    \end{equation*}
    From the $L^{p}$ convergence of $f_{n}$ to $f$, we know that $\{f_{n}\}_{n\ge1}\subseteq L^{p}(M)$ is a Cauchy sequence in $L^{p}(M)$, and so $\{v^{(n)}\}_{n\ge1}\subseteq H^{p}(M)$ is a Cauchy sequence in $H^{p}(M)$. From the completeness of $H^{p}(M)$ it follows that there exists a function $v\in H^{p}(M)$ such that $v^{(n)}\rightarrow v$ in $H^{p}(M)$, and so $v$ is a solution to the desired equation. In particular, since $v^{(n)}\ge0$ for every $n$ we have that $v\ge0$ almost everywhere.
\end{proof}

\begin{remark}\label{Rmk_p-Resolvent}
 The previous result can also be deduced by a purely operator-theoretic argument. Indeed, denoting by $-\Delta^{(p)}$ the $p$-realization of the Friedrichs Laplacian (see, for instance, \cite{hinz2023removable, grigoryan2009heat, dafermos2004handbook, GM}) and by $P^{(p)}_t$ the associated semigroup, as proved in \cite[Theorem 3.7]{ohta2009heat} one has
    \begin{align}
        \norm{P^{(p)}_t}{\mathcal{L}(L^p)}\leq e^{-\frac{4(p-1)}{p^2} \lambda_1(M) t}=e^{-\frac{4}{pp'} \lambda_1(M) t}, 
    \end{align}
    where $\norm{\cdot }{\mathcal{L}(L^p)}$ is denoting the operator norm of a linear operator from $L^p(M)$ to $L^p(M)$ and $\lambda_1(M)$ is the bottom of the spectrum of $-\Delta^{(2)}$. Hence, having fixed any $\lambda>-\frac{4}{pp'}\lambda_1(M)$, for $R_\lambda^{(p)}:=(-\Delta^{(p)}+\lambda)^{-1}$ one has
    \begin{align}
         \norm{R^{(p)}_\lambda}{\mathcal{L}(L^p)}&\leq \int_0^{+\infty} e^{-\lambda t} e^{-\frac{4}{pp'} \lambda_1(M) t} \dint t\\
         &\leq \frac{1}{\lambda+\frac{4}{pp'} \lambda_1(M)}\\
         &=\frac{1}{C_{p,\lambda}},
    \end{align}
    where $C_{p,\lambda}$ is defined as in Proposition \ref{solution with non-neg Lp function}. Thus, defining $v:=R^{(p)}_\lambda f$ it is easy to see that
    \begin{align}
        (-\Delta+ \lambda)v=f
    \end{align}
    in the sense of distributions and
    \begin{align}
        \norm{v}{L^p(M)}=\norm{R^{(p)}_\lambda f}{L^p(M)}\leq \frac{1}{C_{p,\lambda}} \norm{f}{L^p(M)}.
    \end{align}
\end{remark}

\begin{remark}
     We stress that if $(L^{p}$-Liou$_{\lambda-H})$ holds, then the solution constructed in Proposition \ref{solution with non-neg Lp function} is unique. Indeed, let $0\le f\in L^{p}(M)$ and $v,w\in H^{p}(M)$ satisfying
     \begin{equation*}
         (-\Delta+\lambda)v = f, \quad\quad
         (-\Delta+\lambda)w = f.
     \end{equation*}
     Then $u:=v-w$ is an $L^{p}$ solution to 
     \begin{equation*}
         (-\Delta+\lambda)u= 0,
     \end{equation*}
     so that $u\equiv 0$ by the $L^{p}$-Liouville property for $\lambda$-harmonic functions. In particular, it follows $v\equiv w$.
\end{remark}


\section{Positivity preservation and essential self-adjointness}\label{Sec_LpPP_ESA_Core_CSG}
As previously mentioned, the primary objective of this work is to strengthen the connections between certain functional analytic properties of Schr\"odinger operators.  Specifically, we’ll explore the following chain of implications:
\gbox{\begin{align}
    (L^{2}\textnormal{-PP})\Longrightarrow (ESA_{V}) \Longrightarrow (ESA_{\lambda})\Longleftrightarrow (ESA),
\end{align}}
\noindent where $V\in L^{2}_{loc}(M)$ is a general potential  bounded from below and $\lambda\in\mathbb{R}$. Analogous implications for the case $p\neq 2$ are also considered in the present section. 
\gbox{\begin{align}
    (L^{p}\textnormal{-PP})\Longrightarrow (Core_{V}) \Longrightarrow (Core_{\lambda})\Longleftrightarrow (Core),
\end{align}}
Even if most of the following results are well-known to experts, in particular for $p=2$, some connections with the contraction semigroup property seem to be hidden in the various references, so we have decided to highlight them and to present some proofs. 

\subsection{From ($L^2$-PP) to ($ESA_{V}$)}\ \\
We start from the first implication, whose proof is contained in the following 

\begin{proposition}[\cite{BMS}]\label{prop_L2pp_ESA-D+V}
    Let $(M,g)$ be a (possibly incomplete) Riemannian manifold, and fix $C\leq V\in L^2_{loc}(M)$, where $C\in \rr{}$. If $M$ is $(L^{2}$-\textnormal{PP}$)$, then it is $(ESA_{V})$.
\end{proposition}
\begin{proof}
    Set $\widetilde{V}:=V-C\geq 0$. Thanks to Proposition \ref{prop_invariance_translation}, it is sufficient to prove the essential self-adjointness of the operator $Q = -\Delta + \widetilde{V} + 1$.  Since $Q\ge 1$ in the sense of quadratic forms, Theorem \ref{thm_RS2_X26} tells us that we only need to show that the unique $L^2$ distributional solution to $Qu = 0$ is the constant zero function.  Therefore, given any $u \in L^2(M)$ satisfying $Qu = 0$, we have
    \begin{align}
        0=Qu=(-\Delta+\widetilde{V}+1)u,
    \end{align}
    and hence
    \begin{align}
        \Delta u= (1+\widetilde{V})u.
    \end{align}
    By Theorem \ref{Thm_Brezis-Kato}
    \begin{align}
        \Delta u_+ \geq (1+\widetilde{V})u_+ \geq u_+,
    \end{align}
    and thus
    \begin{align}
        (-\Delta+1)(-u^+)\geq 0,
    \end{align}
    which implies $u^+\leq 0$ by the ($L^2$-PP). Whence $u^+=0$ and $u\leq 0$. By replacing $u$ with $-u$, we similarly get $u\geq 0$, obtaining $u\equiv 0$.
\end{proof}

\noindent A straightforward consequence of Proposition \ref{prop_invariance_translation} and the previous result is the next

\begin{corollary}\label{Cor_ESA_V_ESA_lambda}
    Let $(M,g)$ be a (possibly incomplete) Riemannian manifold.  If $(M,g)$ is $(ESA_{V})$ for every $V\in L^2_{loc}(M)$ bounded from below, then $(M,g)$ is $(ESA_{\lambda})$ for every $\lambda \in L^\infty(M)$. In particular, $(M,g)$ is $(ESA_{\lambda})$ for every $\lambda\in \rr{}$.
\end{corollary}

\subsubsection*{The case $p\ne2$}\label{subsec_core-semigroup-p different from 2} 
In the case $p\ne 2$, there are two natural properties associated to the $(L^{p}$-PP$)$ property to look at: the fact that the minimal extension of any Schrodinger operator with non-negative potential in $L^{p'}_{loc}(M)$ generates a contraction semigroup in $L^{p'}(M)$ (note that the conjugate exponent $p'$ must be considered) and the fact that $C^\infty_c(M)$ is a core for the maximal extension of $-\Delta+V$ (i.e. the $p$-maximal extension). See \cite{kato1986lp, BV24, GP19, milatovic2010m, MT16, PVV} and the references therein. 

As we shall see in the next sections, for an operator between $L^{p}$ spaces, the property of being the generator of a contraction semigroup is, in fact, weaker than that of $C^\infty_c(M)$ being a core for the maximal extension; see Example \ref{contract but not core}. In the particular case $p=2$, these two properties collapse into the same one, hence making them indistinguishable. See Proposition \ref{prop_semigroups equivalent cores}.

We start by observing the following relation between the core property for the $p$-maximal and the $p'$-maximal extensions of $-\Delta+V$, and by giving a characterization of the core property in terms of the semigroups generated by the minimal extensions.

\begin{proposition}\label{prop_semigroups equivalent cores}
    Let $(M,g)$ be a (possibly incomplete) Riemannian manifold, let $p\in(1,+\infty)$ and let $0\le V\in L^{p}_{loc}(M)\cap L^{p'}_{loc}(M)$. Then, the following are equivalent:
    \begin{enumerate}
        \item $C^{\infty}_{c}(M)$ is a core for $(-\Delta+V)_{p,max}$;
        \item $C^{\infty}_{c}(M)$ is a core for $(-\Delta+V)_{p',max}$;
        \item both $(-\Delta+V)_{p,min}$ and $(-\Delta+V)_{p',min}$ generate a contraction semigroup.
    \end{enumerate}
\end{proposition}

\begin{proof}
\underline{$(1)\Leftrightarrow (2)$}: Assume that $C^{\infty}_{c}(M)$ is a core for $(-\Delta+V)_{p,max}$ in $L^{p}(M)$, so that $(-\Delta+V)_{p,max}=(-\Delta+V)_{p,min}$. By the relation between the $p$-maximal extension and the $p'$-minimal extension, we have (see \cite{Ka13}, Theorem 5.29)
    \begin{equation*}
        \begin{split}
            (-\Delta+V)_{p',max}&=(-\Delta+V)_{p,min}^{*}=(-\Delta+V)_{p,max}^{*}\\
            &=(-\Delta+V)_{p'  ,min}^{**}=(-\Delta+V)_{p',min},
        \end{split}
    \end{equation*}
and so the core property for $p'$.

\underline{$(1)\Leftrightarrow (3)$:} Mimicking the arguments in \cite{milatovic2010m}, the steps of the proof are the following:
    \begin{itemize}
        \item since $(-\Delta+V)_{p,min}$ generates a contraction semigroup, it is $m$-accretive by Theorem \ref{t:X.48} and the remark before \cite[Theorem X.49]{RS2};
        \item since $L^{p}(M)$ is reflexive, the contraction property is inherited by the adjoint (see Remark \ref{Rmk_adjoint norms}), so that $(-\Delta+V)_{p,max}=(-\Delta+V)_{p',min}^*$ generates a contraction semigroup, and hence is $m$-accretive;
        \item by the $m$-accretivity of $(-\Delta+V)_{p,min}$, it then follows that $(-\Delta+V)_{p,min}=(-\Delta+V)_{p,max}$.
    \end{itemize}

Conversely, suppose that $C^\infty_c(M)$ is a core for both $(-\Delta+V)_{p,max}$ and $(-\Delta+V)_{p',max}$. Since $(-\Delta+V)_{p,min}$ and $(-\Delta+V)_{p',min}$ are both accretive, it then follows that
    \begin{align}
        (-\Delta+V)_{p,max}^*=(-\Delta+V)_{p',min}
    \end{align}
    and
    \begin{align}
        (-\Delta+V)_{p',max}^*=(-\Delta+V)_{p,min}
    \end{align}
    are both accretive. By the Corollary to Theorem \ref{t:X.48}, we conclude that both $(-\Delta+V)_{p,max}=(-\Delta+V)_{p,min}$ and $(-\Delta+V)_{p',max}=(-\Delta+V)_{p',min}$ generate a contraction semigroup.
\end{proof}

\begin{remark}\label{Rmk_adjoint norms}
    We observe that on a reflexive Banach space $(X,\norm{\cdot}{X})$, a densely defined, closed, linear operator $T$ generates a (strongly continuous) contraction semigroup if and only if $T^*$ generates a (strongly continuous) contraction semigroup. Indeed, on reflexive spaces we have the identity
    \begin{equation}\label{norm_of_operator_and_adj}
        \Vert A \Vert_{\mathcal{L}(X)} = \Vert A^{*} \Vert_{\mathcal{L}(X^{*})}
    \end{equation}
    for every bounded linear operator $A$ on $X$. Moreover, by Corollary 10.6 in \cite{pazy2012semigroups} we know that the adjoint of the semigroup generated by $T$ is the (strongly continuous) semigroup generated by $T^{*}$, and the claim follows by applying  \eqref{norm_of_operator_and_adj} to the semigroups generated by $T$ and $T^{*}$ and using the definition of contraction semigroup. 
\end{remark}

In the next step we are going to relate the positivity preservation to the core property for the operator $-\Delta+V$. As a first stage, we show that the ($L^{p'}$-PP) is strong enough to ensure that $(-\Delta+V)_{p,min}$ generates a contraction semigroup in $L^{p}(M)$.

\begin{proposition}\label{prop_LpPP_CSG-D+V}
    Let $(M,g)$ be a (possibly incomplete) Riemannian manifold, let $p\in(1,+\infty)$ and let $0\leq V\in L^{p}_{loc}(M)$. If $M$ is $(L^{p'}$-\textnormal{PP}$)$, then $(-\Delta+V)_{p,min}$ generates a contraction semigroup in $L^{p}(M)$. Moreover, if $V\in L^{p}_{loc}(M)\cap L^{p'}_{loc}(M)$ then $(-\Delta+V)_{p',max}$ generates a contraction semigroup in $L^{p'}(M)$.
\end{proposition}
\begin{proof}
    We provide a short sketch of the proof. 
    
    Fix $\lambda>0$. Thanks to the ($L^{p'}$-PP), following the argument of \cite[Section 2]{milatovic2010m}, one can observe that $\textnormal{Ran}\left((-\Delta+V+\lambda)|_{C^\infty_c(M)}\right)$ is dense in $L^{p}(M)$ and hence that $\textnormal{Ran}\left((-\Delta+V)_{p,min}+\lambda\right)=L^{p}(M)$. Since $(-\Delta+V)_{p,min}$ is accretive in $L^{p}(M)$, see \cite[Remark 2.2]{milatovic2010m}, one concludes thanks to Theorem \ref{t:X.48}.
    
    For the last statement, supposing also that $V\in L^{p'}_{loc}(M)$ and observing that
    \begin{align}
        (-\Delta+V)_{p',max}=(-\Delta+V)_{p,min}^*,
    \end{align}
    arguing as in \cite[Theorem 4.12]{BV24} it follows that $(-\Delta+V)_{p',max}$ generates a contraction semigroup.
\end{proof}

\noindent Concerning the $p$-maximal extension, the contraction semigroup property requires both the $L^{p}$ and $L^{p'}$-positivity preserving properties. We have the following

\begin{proposition}\label{prop_LpPP_CSG-D+V_p,max}
    Let $(M,g)$ be a (possibly incomplete) Riemannian manifold, let $p\in(1,+\infty)$ and let $0\leq V\in L^{p}_{loc}(M)$. If $M$ is $(L^{p}$-\textnormal{PP}$)$ and $(L^{p'}$-\textnormal{PP}$)$, then $(-\Delta+V)_{p,max}$ generates a contraction semigroup in $L^{p}(M)$.
\end{proposition}
\begin{proof}
    Again, we provide a short sketch of the proof, following the argument of \cite[Section 3]{milatovic2010m}. 
    
    First, using the Kato inequality one proves that $(-\Delta+V)_{p,max}$ is a closed operator. Moreover, since the domain of $(-\Delta+V)_{p,min}$ is contained in that of $(-\Delta+V)_{p,max}$, and from the fact that $\textnormal{Ran}\left((-\Delta+V)_{p,min}+\lambda\right)=L^{p}(M)$ for every $\lambda>0$, it follows that $\textnormal{Ran}\left((-\Delta+V)_{p,max}+\lambda\right)=L^{p}(M)$ for every $\lambda>0$.
    
    Secondly, thanks to the ($L^p$-PP) one obtains an estimate of the form
    \begin{equation*}
        \lambda \Vert u\Vert_{L^{p}(M)}\le \Vert ((-\Delta+V)_{p,max}+\lambda)u\Vert_{L^{p}(M)}
    \end{equation*}
    for every function $u$ in the domain of $(-\Delta+V)_{p,max}$. In particular, this implies the injectivity of the operator $(-\Delta+V)_{p,max}+\lambda$.
    
    Thus, $(-\infty,0)$ is contained in the resolvent set of $(-\Delta+V)_{p,max}$, and the resolvent $((-\Delta+V)_{p,max}+\lambda)^{-1}$ is a bounded linear operator defined on the whole $L^{p}(M)$ and satisfying the estimate
    \begin{equation*}
        \Vert ((-\Delta+V)_{p,max}+\lambda)^{-1}\Vert_{\mathcal{L}(L^{p})}\le \frac{1}{\lambda}
    \end{equation*}
    for every $\lambda>0$. Using the Hille-Yosida theorem, we conclude that $(-\Delta+V)_{p,max}$ generates a contraction semigroup.
\end{proof}

\begin{remark}
    Note that in the particular case of $p\geq 2$, in order to deduce that $(-\Delta+V)_{p,max}$ generates a contraction semigroup, it is enough to assume only the validity of the ($L^p$-PP), without requiring the ($L^{p'}$-PP). This is a consequence of the Proposition \ref{prop_LpPP_CSG-D+V}.
\end{remark}

We conclude by observing that, thanks to Proposition \ref{prop_LpPP_CSG-D+V} and Proposition \ref{prop_LpPP_CSG-D+V_p,max}, requiring both the ($L^p$-PP) and ($L^{p'}$-PP) is sufficient for the set of compactly supported smooth functions to be a core for $(-\Delta+V)_{p,max}$, recovering what happens in the case $p=2$.

\begin{corollary}\label{prop_LpPP+Lp'PP_Core}
    Let $(M,g)$ be a (possibly incomplete) Riemannian manifold, let $p\in(1,+\infty)$ and let $0\leq V\in L^{p}_{loc}(M)$. If $M$ satisfies $(L^{p}$-\textnormal{PP}$)$ and $(L^{p'}$-\textnormal{PP}$)$, then $C^{\infty}_{c}(M)$ is a core for $(-\Delta+V)_{p,max}$.
\end{corollary}
\begin{proof}
    By Proposition \ref{prop_LpPP_CSG-D+V} and Proposition \ref{prop_LpPP_CSG-D+V_p,max}, both $(-\Delta+V)_{p,min}$ and $(-\Delta+V)_{p,max}$ generate a contraction semigroup, hence they are m-accretive. By maximality and using the fact that $(-\Delta+V)_{p,max}$ is an extension of $(-\Delta+V)_{p,min}$, the equality of the two operators follows, and so the claim.
\end{proof}


\section{Liouville properties and positivity preservation}\label{Sec_class_Lp_Liou_Monot}
After having established a first sequence of implications primarily involving functional analytic properties, we now explore the relationships between various notions of $L^2$-Liouville properties for $\lambda$-subharmonic distributions and the positivity preservation. This connection with the $(L^2$-PP$)$ serves as a bridge between essential self-adjointness and Liouville properties. Specifically, we prove the following chain of implications:
\begin{center}
    \gbox{\begin{align}
    \begin{array}{ccccccc}
    L^{2}\textnormal{-Liou}_{SH} &
    \Leftrightarrow &
    L^{2}\textnormal{-Liou}_{\lambda-SH} & \Rightarrow &  L^{2}\textnormal{-Liou}_{\lambda-H} & \Rightarrow &  L^{2}\textnormal{-Liou}_{H}\\
    & & \Updownarrow & & \Updownarrow & &\\
    & & L^2\textnormal{-PP} & \Rightarrow & ESA. & & \\
    \end{array}
\end{align}}
\end{center}

\noindent Moreover, we show that, under certain assumptions, the arrow between ($L^{2}$-Liou$_{\lambda-SH}$) and ($L^{2}$-Liou$_{\lambda-H}$) can be reversed.

Finally, the situation in the $p\ne2$ case will be discussed.
\begin{center}
    \gbox{\begin{align}
    \begin{array}{ccccccc}
    L^{p}\textnormal{-Liou}_{SH} &
    \Leftrightarrow &
    L^{p}\textnormal{-Liou}_{\lambda-SH} & \Rightarrow &  L^{p}\textnormal{-Liou}_{\lambda-H} & \Rightarrow &  L^{p}\textnormal{-Liou}_{H}\\
    & & \Updownarrow & & \Updownarrow & &\\
    & & L^p\textnormal{-PP} & \Rightarrow & CSG_{-\Delta_{p',min}}. & & \\
    \end{array}
\end{align}}
\end{center}

\subsection{Relation between ($L^{2}\textnormal{-Liou}_{SH}$), ($L^{2}\textnormal{-Liou}_{\lambda-SH}$) and ($L^2$-PP) - Part I}\ \\
For fixed $\bar{\lambda}\in\mathbb{R}$, we observe that the implication 
\begin{equation*}
    (L^2-\textnormal{Liou}_{\bar{\lambda}-SH}) \Rightarrow (L^{2}\textnormal{-Liou}_{\lambda-SH})
\end{equation*} 
for all $\lambda \geq \bar{\lambda}$ is a direct consequence of the definitions. Since the proofs in the case $p=2$ and $p\ne2$ are exactly the same, we state a result that includes both cases. 

\begin{lemma}\label{lem_monotonicity_L_l-SH}
    Let $(M,g)$ be a (possibly incomplete) Riemannian manifold and $p\in (1,+\infty)$. If there exists $\bar{\lambda}\in\mathbb{R}$ so that $M$ is \textnormal{(}$L^{p}\textnormal{-Liou}_{\bar{\lambda}-SH}$\textnormal{)}, then $M$ is \textnormal{(}$L^{p}\textnormal{-Liou}_{\lambda-SH}$\textnormal{)} for every $\lambda\geq \bar{\lambda}$.
\end{lemma}
\begin{proof}
Assume that $(M,g)$ is ($L^{p}\textnormal{-Liou}_{\bar{\lambda}-SH}$). Having fixed $\lambda>\bar{\lambda}$ and $0\leq u\in L^p(M)$ satisfying $(-\Delta+\lambda)u\leq 0$ in the sense of distributions, then $u$ also satisfies $(-\Delta +\bar{\lambda}) u \leq 0$ in the sense of distributions. By ($L^{p}\textnormal{-Liou}_{\bar{\lambda}-SH}$), it follows that $u$ has to be constant, implying the validity of the ($L^{p}\textnormal{-Liou}_{\lambda-SH}$) for every $\lambda>\bar{\lambda}$.
\end{proof}

\noindent Hence, one can easily get the following

\begin{corollary}\label{cor_l_sh_l_lambda_sh}
     Let $(M,g)$ be a (possibly incomplete) Riemannian manifold and $p\in (1,+\infty)$. If $M$ is \textnormal{(}$L^{p}\textnormal{-Liou}_{SH}$\textnormal{)}, then it is \textnormal{(}$L^{p}\textnormal{-Liou}_{\lambda-SH}$\textnormal{)} for every $\lambda\geq 0$.
\end{corollary}

The next step is to highlight the relationship between ($L^{2}\textnormal{-Liou}_{\lambda-SH}$) and ($L^2$-PP) by showing that these two properties are equivalent. We begin with the following
\begin{proposition}\label{prop_L_l-SH_and_L^2-PP}
    Let $(M,g)$ be a (possibly incomplete) Riemannian manifold and fix $\lambda > 0$. Then, $M$ is \textnormal{(}$L^{p}\textnormal{-Liou}_{\lambda-SH}$\textnormal{)} if and only if it is $(L^{p}$\textnormal{-PP}$_{\lambda})$. If the volume of $M$ is infinite, the same conclusion holds for every $\lambda\in\mathbb{R}$.
\end{proposition}
\begin{proof}
    Assume that $M$ satisfies ($L^{p}\textnormal{-Liou}_{\lambda-SH}$), and let $u\in L^{p}(M)$ be a distributional solution to $(-\Delta+\lambda)u\ge0$. If we set $v:=-u$, then $v\in L^{p}(M)$ is a distributional solution to $\Delta v \ge \lambda v$, and by the Brezis-Kato inequality we have that $v_{+}\in L^{p}(M)$ is a non-negative, distributional solution to the same differential inequality. By the ($L^{p}\textnormal{-Liou}_{\lambda-SH}$) property, we can conclude that $u_{-}=v_{+}\equiv const.\geq 0$. Hence, either $u\equiv -u_-$ or $u\geq 0$. If $M$ has infinite volume, the only constant function in $L^p(M)$ is the null function, providing $u_-\equiv 0$. If $M$ has finite volume, then $\lambda>0$ and by $(-\Delta+\lambda)u_-\leq 0$ it follows that $u_- \equiv 0$. Thus, in both the cases, $u \geq 0$. 
    
    Conversely, assume that $M$ is $(L^{p}$-PP$_{\lambda})$ and consider $u\in L^{p}(M)$ a non-negative, distributional solution to $\Delta u\ge \lambda u$ for some $\lambda\in\mathbb{R}$. Then, the function $v:=-u$ satisfies
    \begin{equation*}
        \begin{cases}
            (-\Delta+\lambda)v\ge0 & \text{in $\mathcal{D}'(M)$}\\
            v\in L^{p}(M)
        \end{cases}
    \end{equation*}
    and by the $(L^{p}$-PP$_{\lambda})$ property we have $v\ge0$, which means $u\le0$, so $u$ is the null function.
\end{proof}

\noindent As observed in Lemma \ref{lem_monotonicity_L_l-SH}, if ($L^{p}\textnormal{-Liou}_{\bar{\lambda}-SH}$) holds for a fixed $\bar{\lambda}$, then it holds for every $\lambda\ge\bar{\lambda}$. By combining Lemma \ref{lem_monotonicity_L_l-SH} with Proposition \ref{prop_L_l-SH_and_L^2-PP}, we also get

\begin{corollary}\label{cor_translation L^p-PP}
    Let $(M,g)$ be a (possibly incomplete) Riemannian manifold. If there exists $\bar{ \lambda} \in \rr{}$ so that $M$ is $(L^p$\textnormal{-PP}$_{\bar{\lambda}})$, then $(M,g)$ is $(L^p$\textnormal{-PP}$_{\lambda})$ for every $\lambda\geq \bar{\lambda}$.
\end{corollary}

\noindent We now show that also the reversed monotonicity for $\lambda$ holds up to a certain threshold, which depends on the value of $p$ and the bottom of the spectrum of $-\Delta$.

\begin{proposition}\label{lambda tra 0 e 1}
    Let $(M,g)$ be a (possibly incomplete) Riemannian manifold, $p\in(1,+\infty)$, and
    \begin{equation*}
        C_{p}(M):=\frac{4}{pp'}\lambda_{1}(M),
    \end{equation*}
    where $\lambda_{1}(M)\ge0$ is the bottom of the spectrum of $-\Delta$ in $L^{2}(M)$. If $M$ is $(L^{p}$\textnormal{-PP}$_{\bar{\lambda}})$ for some $\bar{\lambda}>-C_p(M)$, then it is $(L^{p}$\textnormal{-PP}$_{\lambda})$ for every $\lambda\in(-C_p(M),\bar{\lambda})$, and hence for every $\lambda\in (-C_p(M), +\infty)$.
\end{proposition}

\begin{proof} By monotonicity with respect to $\lambda$, we can assume $\bar{\lambda}>0$. Moreover, without loss of generality we can assume $\bar{\lambda}=1$. In the general case the proof follows by simply readapting the argument. 

By contradiction, assume that there exist a $\lambda\in(-C_p(M),1)$ and a function $u\in L^p(M)$ such that
    \begin{equation*}
        \begin{cases}
            (-\Delta+\lambda)u\ge0 &\text{in $\mathcal{D}'(M)$,}\\
            u_{-}\not\equiv0.
        \end{cases}
    \end{equation*}
    Note that 
    \begin{equation*}
        \begin{split}
            (-\Delta+1)u&=(-\Delta+\lambda)u+(1-\lambda)u\\
            &\ge-(1-\lambda)u_{-}
        \end{split}
    \end{equation*}
    in $\mathcal{D}'(M)$, and hence
    \begin{equation}\label{equ u e u-}
        (-\Delta+1)u+(1-\lambda)u_{-}\ge 0
    \end{equation}
    in $\mathcal{D}'(M)$. Since $u_{-}\in L^{p}(M)$, we can apply Proposition \ref{solution with non-neg Lp function} to obtain a non-negative, non-identically zero $v\in H^{p}(M)$ which solves the equation
    \begin{equation}\label{eq v per L2PP}
        (-\Delta+1)v=u_{-}
    \end{equation}
    and satisfying the estimate
    \begin{equation}\label{stima basso u-}
        (C_p(M)+1)\Vert v\Vert_{L^{p}(M)}\le\Vert u_{-}\Vert_{L^{p}(M)}.
    \end{equation}
    Setting $w:=u+(1-\lambda)v\in L^{p}(M)$, by \eqref{equ u e u-} we have
    \begin{equation*}
        \begin{split}
            (-\Delta+1)w&= (-\Delta+1)u + (1-\lambda)(-\Delta+1)v \\
            &=(-\Delta+1)u+(1-\lambda)u_{-} \\
            &\ge 0
        \end{split}
    \end{equation*}
    in $\mathcal{D}'(M)$, and by the $(L^{p}$-PP$)$ we have $w\ge0$. Now, we will get a contradiction by showing that $w_{-}\not\equiv0$. First, note that $w=u_{+}+(1-\lambda)v\ge0$ on $\{u>0\}$, and that $w=-u_{-}+(1-\lambda)v$ on $\{u<0\}$ (which has positive measure by assumption), so that 
    \begin{equation*}
        w_{-}=(-u_{-}+(1-\lambda)v)_{-}.
    \end{equation*}
    By contradiction, assume that $w_{-}\equiv0$ on $M$. Then 
    \begin{equation*}
        (1-\lambda)v\ge u_{-}\ge0 \quad \text{on $M$,}
    \end{equation*}
    so that 
    \begin{equation*}
        \Vert u_{-}\Vert_{L^{p}(M)}\le (1-\lambda)\Vert v\Vert_{L^{p}(M)}.
    \end{equation*}
    Combining this with \eqref{stima basso u-} we get
    \begin{equation*}
        (C_p(M)+1)\Vert v \Vert_{L^{p}(M)}\le\Vert u_{-} \Vert_{L^{p}(M)}\le(1-\lambda)\Vert v \Vert_{L^{p}(M)},
    \end{equation*}
    which implies
    \begin{equation*}
        \lambda\le -C_p(M),
    \end{equation*}
    contradiction. Then , we must have $w_{-}\not\equiv0$, but this contradicts the $(L^{p}$-PP$)$.
\end{proof}

\begin{remark}
    As we will see in Corollary \ref{cor_equivalence_L2SH_Liouville_L2PP}, $(L^{p}$-PP$_{\lambda})$ with $\lambda>0$ and $(L^{p}$-Liou$_{SH})$ are actually equivalent. However, Proposition \ref{lambda tra 0 e 1} is still interesting because it gives the $L^p$-positivity preservation property for $-\Delta+\lambda$ also for a range of negative values of $\lambda$ which is related to the bottom of the spectrum of $-\Delta$. 
\end{remark}

\noindent Summing up the previous discussion, we get

\begin{corollary}\label{cor_equivalence_L2PP_Liouville}
    Let $(M,g)$ be a (possibly incomplete) Riemannian manifold and let $p\in (1,+\infty)$. Then, the following are equivalent:
    \begin{enumerate}
        \item $(L^{p}$\textnormal{-PP}$_\lambda)$ holds for every $\lambda>0$;
        \item $(L^{p}$\textnormal{-PP}$_\lambda)$ holds for some $\lambda>0$;
        \item $(L^{p}\textnormal{-Liou}_{\lambda-SH})$ holds for every $\lambda>0$;
        \item $(L^{p}\textnormal{-Liou}_{\lambda-SH})$ holds for some $\lambda>0$.
    \end{enumerate}
    If the volume of $M$ is infinite, the same conclusion holds if we replace $\lambda>0$ with $\lambda>-C_p(M)$.
\end{corollary}


\subsection{From ($L^{2}\textnormal{-Liou}_{\lambda-SH}$) to ($L^{2}\textnormal{-Liou}_{\lambda-H}$)}\ \\
It is straightforward that for every positive $\lambda$, ($L^{p}\textnormal{-Liou}_{\lambda-SH}$) implies ($L^{p}\textnormal{-Liou}_{\lambda-H}$) for every $p\in(1,+\infty)$.
\begin{lemma}\label{lem_l-SH_implies_l-H}
    Let $(M,g)$ be a (possibly incomplete) Riemannian manifold, let $p\in(1,+\infty)$ and fix $\lambda\in \rr{}$. If $M$ is $(L^{p}\textnormal{-Liou}_{\lambda-SH})$, then it is $(L^{p}\textnormal{-Liou}_{\lambda-H})$.
\end{lemma}
\begin{proof}
Assume that $(M,g)$ is ($L^{p}\textnormal{-Liou}_{\lambda-SH}$). Having fixed $u\in L^p(M)$ so that $(-\Delta +\lambda)u=0$ in $M$, by the Brezis-Kato inequality we have
\begin{align}\label{app_3}
    (-\Delta  + \lambda)u_+\leq 0 \quad \inn M,
\end{align}
which implies that $u_+$ is constant by the ($L^{p}\textnormal{-Liou}_{\lambda-SH}$). By replacing $u$ with $-u$, we similarly get $u_-\equiv const$. Hence $u$ is constant.
\end{proof}

It is not yet clear to us if the properties ($L^{2}\textnormal{-Liou}_{\lambda-SH}$) and ($L^{2}\textnormal{-Liou}_{\lambda-H}$) are equivalent.  While equivalence is generally expected to be false (see \cite[Chapter X1V.5]{guneysu2017covariant}), using the language of Dirichlet forms, one can prove that if $-\Delta$ is essentially self-adjoint then ($W^{1,2}$-PP) holds and so does ($W^{1,2}\textnormal{-Liou}_{\lambda-SH}$), providing the equivalence between ($W^{1,2}\textnormal{-Liou}_{\lambda-SH}$) and ($L^{2}\textnormal{-Liou}_{\lambda-H}$). Indeed, the essential self-adjointness implies that the spaces $W^{1,2}(M)$ and $W^{1,2}_{0}(M)$ are equal, and therefore that $W^{1,2}(M)$ is contained in the domain of the Dirichlet form associated to $(-\Delta,W^{1,2}(M))$. Positivity preserving then follows from Theorem 2.2.1 in \cite{FOT11}, pp. 77–78. However, the problem remains open in the general case of nonnegative 1-subharmonic distributions in $L^2(M)$, which, a priori, are only $W^{1,2}_{loc}$ (Theorem \ref{Reg_sub_dist}).

A sufficient condition for the validity of the converse implication to Lemma \ref{lem_l-SH_implies_l-H} is given in the following 

\begin{proposition}\label{Prop_positive cone}
    Let $(M,g)$ be a (possibly incomplete) Riemannian manifold. Assume that the cone $C^\infty_c(M)_+:=\{\phi\in C^\infty_c(M) : \phi\geq 0\}$ is dense in the cone $H^2_+:=\{u\in H^2(M) : u\geq 0\ \textnormal{a.e.}\}$ with respect to the graph norm $\norm{\cdot}{H^2(M)}$.
    If $M$ is $(L^{2}\textnormal{-Liou}_{\lambda-H})$ for every $\lambda>0$, then it is $(L^{2}\textnormal{-Liou}_{\lambda-SH})$ for every $\lambda>0$.
\end{proposition}

\begin{proof} Thanks to Corollary \ref{cor_equivalence_L2PP_Liouville}, we can equivalently prove that $(M,g)$ is $(L^2$-PP$_\lambda)$ for every $\lambda>0$. By the same corollary, we can fix $\lambda=1$. Let us denote by $R:L^2(M)\to H^2(M)$ the resolvent of $(-\Delta+1, H^2(M))$: since the validity of the $L^2_{1-H}$-Liouville property is equivalent to the essential self-adjointness of $(-\Delta+1, C^\infty_c(M))$, see Proposition \ref{prop_L_l-H_ESA_-D+l} below, it follows that $(-\Delta+1, H^2(M))$ coincides with the Friedrichs extension of $(-\Delta+1,C^\infty_c(M))$, which is positivity preserving. Hence, as already noticed, $R$ is positivity preserving. 

Now consider $u\in L^2(M)$ so that $(-\Delta+1)u\geq 0$ in the sense of distributions. For fixed $0\leq \phi\in C^\infty_c(M)$, set $0\leq \eta:=R\phi \in H^2(M)$. Since $\eta$ satisfies $(-\Delta+1)\eta=\phi \in C^\infty_c(M)$, by elliptic regularity it follows that $\eta \in C^\infty(M)$. By the assumption, there exists a sequence $\{\eta_k\}_k\subset C^\infty_c(M)$ of nonnegative functions such that $\eta_k \to\eta$ in the graph norm $\norm{\cdot}{H^2(M)}$, i.e.
        \begin{align}
            \norm{\eta-\eta_k}{L^2(M)}\to 0 \quad \andd \quad \norm{\Delta \eta- \Delta \eta_k}{L^2(M)}\to 0.
        \end{align}
        Hence
        \begin{align}
            \int_M u\phi \dvol &= \int_M u (-\Delta+1)\eta \dvol\\
            &=\lim_{k\to +\infty} \underbrace{\int_M u (-\Delta+1)\eta_k \dvol}_{\geq 0\ \textnormal{by dist. ineq. for }u}\\
            &\geq 0.
        \end{align}
        Since it holds for every $0\leq \phi \in C^\infty_c(M)$, it follows that $u\geq 0$.
\end{proof}

\begin{remark}
    It would be interesting to determine under which geometric or functional conditions (or possibly to characterize when) the density of $C^\infty_c(M)_+$ in $H^2_+$ assumed in Proposition \ref{Prop_positive cone} is satisfied.
\end{remark}

\subsection{Equivalence between ($L^{2}\textnormal{-Liou}_{\lambda-H}$) and $(ESA)$}\ \\
The fact that ($L^{2}\textnormal{-Liou}_{\lambda-H}$) is equivalent to the essential self-adjointness of the operator $-\Delta+\lambda$ is a consequence of the functional analytic result of Theorem \ref{thm_RS2_X26}.

\begin{proposition}\label{prop_L_l-H_ESA_-D+l}
    Let $(M,g)$ be a (possibly incomplete) Riemannian manifold. Then, for any fixed $\lambda>0$, $(M,g)$ is $(L^{2}\textnormal{-Liou}_{\lambda-H})$ if and only if $(M,g)$ is $(ESA_{\lambda})$. If the volume of $M$ is infinite, the same conclusion holds for every $\lambda>-\lambda_1(M)$. 
\end{proposition}
\begin{proof}
    Observe that for any fixed $\lambda >0$ (resp. $\lambda>-\lambda_1(M)$ if $M$ has infinite volume)  the operator $-\Delta_{2,min}+\lambda$ is positive. Hence, by Theorem \ref{thm_RS2_X26} the property $(ESA_{\lambda})$ is equivalent to the fact that $ker\left((-\Delta_{2,min}+\lambda)^*\right)=\{0\}$, i.e., since $(-\Delta_{2,min}+\lambda)^*=-\Delta_{2,max}+\lambda$, to
    \begin{align}
        \begin{cases}
            (-\Delta_{2,max}+\lambda)u=0 & \inn M,\\ u\in L^2(M), 
        \end{cases} \quad \Rightarrow \quad u\equiv0,
    \end{align}
    which implies that $(M,g)$ is ($L^{2}\textnormal{-Liou}_{\lambda-H}$).

    Conversely, assume that $(M,g)$ is ($L^2\text{-Liou}_{\lambda-H}$) for some $\lambda >0$ (resp. $\lambda>-\lambda_1(M)$ if $M$ has infinite volume). By the ($L^2\text{-Liou}_{\lambda-H}$) it follows that
    \begin{align}
        \begin{cases}
            (-\Delta_{2,max}+\lambda)u=0 & \inn M,\\ u\in L^2(M), 
        \end{cases} \quad \Rightarrow \quad u\equiv const.
    \end{align}
    If $\lambda\neq 0$, then the unique constant $L^2$-solution to $(-\Delta_{2,max}+\lambda)u=0$ is the constant null function. If $\lambda=0$, then we are in the case of a manifold $M$ with infinite volume and hence the unique $L^2$-integrable function over $M$ is the constant null function. Again by Theorem \ref{thm_RS2_X26}, the claim follows.
\end{proof}

\begin{remark}
    We observe that if $\lambda>-\lambda_1(M)$, $\lambda\neq 0$, then the previous argument works without any restriction on the volume of $M$. The only problematic part, if $M$ has finite volume, is the case $\lambda=0$ for the implication
    \begin{align}
        (L^2\text{-Liou}_{\lambda-H}) \quad \Longrightarrow \quad (ESA_\lambda).
    \end{align}
    Indeed, for $\lambda=0$, the Liouville property for harmonic functions tells us that $u$ is constant, but the equation satisfied by $u$ does not help to ensure that the unique $L^2(M)$ constant satisfying $-\Delta u=0$ is the null function; \cite[Example 3 (ii)]{HMW}.
\end{remark}

\noindent By combining Proposition \ref{prop_L_l-H_ESA_-D+l} with Corollary \ref{cor_equiv_ESA}, we get

\begin{corollary}\label{cor_equiv_L_l-harm}
    Let $(M,g)$ be a (possibly incomplete) Riemannian manifold. Then, the following are equivalent:
    \begin{enumerate}
        \item $(M,g)$ is $(L^{2}\textnormal{-Liou}_{\lambda-H})$ for some $\lambda>0$;
        \item $(M,g)$ is $(L^{2}\textnormal{-Liou}_{\lambda-H})$ for every $\lambda>0$;
        \item $(M,g)$ is $(ESA)$.
    \end{enumerate}
    If the volume of $M$ is infinite, the same conclusion holds if we replace $\lambda>0$ with $\lambda>-\lambda_1(M)$.
\end{corollary}

\subsubsection*{The case $p\ne2$}
When $p\ne2$, the $L^{p}$-Liouville property for $\lambda$-harmonic functions is related to the fact that the $p$-minimal extension of the Laplacian generates a contraction semigroup.  We have the following
\begin{proposition}\label{Lp lambda armomiche e semigruppo}
    Let $(M,g)$ be a (possibly incomplete) Riemannian manifold, $\lambda>0$ and $p\in(1,+\infty)$. Then $(L^{p}\textnormal{-Liou}_{\lambda-H})$ holds if and only if $-\Delta_{p',min}$ generates a contraction semigroup in $L^{p'}(M)$. If the volume of $M$ is infinite, the same conclusion holds for every $\lambda>-C_p(M)$.
\end{proposition}
\begin{proof}
    Assuming the validity of ($L^{p}\textnormal{-Liou}_{\lambda-H}$), arguing as in \cite{milatovic2010m}, one can follow the next steps to prove that $-\Delta_{p',min}$ is accretive and that $\mathrm{Ran}(-\Delta_{p',min}+\lambda)=L^{p'}(M)$ for any fixed $\lambda>-C_p(M)$ (respectively, $\lambda>0$ if $M$ has finite volume):
    \begin{itemize}
        \item by the ($L^{p}\textnormal{-Liou}_{\lambda-H}$) and \cite[Corollary 2.18]{brezis2011functional}, one has that
    \begin{align}
        \overline{\mathrm{Ran}((-\Delta_{p',min}+\lambda)|_{C^\infty_c(M)})}=\mathrm{ker}(-\Delta_{p,max}+\lambda)^\bot=L^{p'}(M),
    \end{align}
    implying that $\mathrm{Ran}((-\Delta_{p',min}+\lambda)|_{C^\infty_c(M)})$ is dense in $L^{p'}(M)$;

    \item having fixed $f\in L^{p'}(M)$, by the previous point there exists $\{u_k\}_k\subset C^\infty_c(M)$ such that $(-\Delta_{p',min} + \lambda)u_k\to f$ in $L^{p'}(M)$. By Remark \ref{Rmk_p-Resolvent}, since the Friedrichs extension of the Laplacian coincides with $-\Delta_{{p'},min}$ on $C^\infty_c(M)$, one has that
    \begin{align}
        (C_p(M)+\lambda)\norm{u_k-u_h}{L^{p'}(M)}\leq \norm{(-\Delta_{p',min} + \lambda)(u_k-u_h)}{L^{p'}(M)}
    \end{align}
    implying that $\{u_k\}_k$ is a Cauchy sequence in $L^{p'}(M)$ and hence $u_k\to u\in L^{p'}(M)$. Since $-\Delta_{p',min}+\lambda$ is closed, $u\in H_0^{p'}(M)$ and $(-\Delta_{p',min}+\lambda)u=f$, implying that
    \begin{align}
        \mathrm{Ran}(-\Delta_{p',min}+\lambda)=L^{p'}(M);
    \end{align}

    \item by the previous point, $(-\Delta_{p',min}+\lambda)$ is surjective. Since $(-\Delta_{p',min}+\lambda)\subseteq (-\Delta^{(p')}+\lambda)$, where $-\Delta^{(p')}$ is the $p'$-realization of the Friedrichs Laplacian, and $(-\Delta^{(p')}+\lambda)$ is injective, it follows that $(-\Delta_{p',min}+\lambda)= (-\Delta^{(p')}+\lambda)$. Indeed, fixed any $v$ in the domain of $(-\Delta^{(p')}+\lambda)$, by the surjectivity of $(-\Delta_{p',min}+\lambda)$ there exists $u\in H^{p'}_0(M)$ such that
    \begin{align}
        (-\Delta^{(p')}+\lambda)v=(-\Delta_{p',min}+\lambda)u=(-\Delta^{(p')}+\lambda)u,
    \end{align}
    where the last identity follows by the fact that $(-\Delta_{p',min}+\lambda)\subseteq (-\Delta^{(p')}+\lambda)$. By the injectivity of $(-\Delta^{(p')}+\lambda)$, it follows that $u=v$ and hence $v\in H^{p'}_0(M)$: thus the domains of $(-\Delta_{p',min}+\lambda)$ and $(-\Delta^{(p')}+\lambda)$ coincide and hence they are the same operator. By Remark \ref{Rmk_p-Resolvent}, the semigroup $P_t^{(p'),\lambda}$ associated to $(-\Delta^{(p')}+\lambda)$ satisfies
    \begin{align}
        \norm{P_t^{(p'),\lambda}}{\mathcal{L}(L^p)}\leq e^{(-\lambda-C_p(M))t}\leq 1,
    \end{align}
    and hence $(-\Delta_{p',min}+\lambda)=(-\Delta^{(p')}+\lambda)$ generates a contraction semigroup.
    \end{itemize}
    
    Conversely, assume that $-\Delta_{p',min}$ generates a contraction semigroup in $L^{p'}(M)$. Hence, $-\Delta_{p',min}$ is $m$-accretive and thus it coincides with $-\Delta^{(p')}$, since $-\Delta^{(p')}$ is accretive. Thus, for any fixed $\lambda>-C_p(M)$ we have that $(-\Delta^{(p')}+\lambda)=(-\Delta_{p',min}+\lambda)$ is bijective, and hence surjective. By Theorem 2.20 in \cite{brezis2011functional} and the fact that $-\Delta_{p',min}^*=-\Delta_{p,max}$, it follows that
    \begin{equation*}
        \{0\} = \mathrm{ker}((-\Delta_{p',min}+\lambda)^{*}) = \mathrm{ker}(-\Delta_{p,max}+\lambda),
    \end{equation*}
    i.e. ($L^{p}\textnormal{-Liou}_{\lambda-H}$) holds.
\end{proof}

\begin{remark}\label{rmk_inv_MasaHW_p_ne_2}
    When the manifold has infinite volume and positive bottom, the previous result gives an analogue of the implication
    \begin{equation*}
        \begin{cases}
            \text{$M$ is ($L^{2}$-Liou$_{H}$)}\\
            \text{Vol($M$)$=\infty$}\\
            \lambda_{1}(M)>0
        \end{cases}\Longrightarrow \textnormal{$-\Delta$ is e.s.a.}
    \end{equation*}
    in \cite[Corollary 4]{HMW} in the case of general $p>1$: indeed, for $p=2$ the fact that $-\Delta_{2,min}$ generates a contraction semigroup in $L^{2}(M)$ is equivalent to the essential self-adjointness.
\end{remark}

\noindent By combining the previous results with Proposition \ref{prop_semigroups equivalent cores}, we get the following
\begin{corollary}\label{cor_core_equivalence_harmonic_liouville}
    Let $(M,g)$ be a (possibly incomplete) Riemannian manifold, $\lambda>0$ and $p\in(1,+\infty)$. Then $C^{\infty}_{c}(M)$ is a core for $-\Delta_{p,max}$ if and only if both ($L^{p}\textnormal{-Liou}_{\lambda-H}$) and ($L^{p'}\textnormal{-Liou}_{\lambda-H}$) hold. If the volume of $M$ is infinite, the same conclusion holds for every $\lambda>-C_p(M)$.
\end{corollary}


\subsection{From ($L^{2}\textnormal{-Liou}_{\lambda-H}$) to ($L^{2}\textnormal{-Liou}_{H}$)}\ \\
Finally, we discuss the relation between $(L^2\text{- Liou}_{\lambda-h})$ (equivalently, $(ESA)$) and ($L^{2}\textnormal{-Liou}_{H}$). The next proposition is contained in \cite[Corollary 4]{HMW}.

\begin{proposition}\label{ESA-->L2H}
Let $(M,g)$ be a (possibly incomplete) Riemannian manifold. The following properties hold: 
\begin{enumerate}
        \item If $(M,g)$ is $(ESA)$, then it is $(L^{2}\textnormal{-Liou}_{H})$.
        \item If $\lambda_1(M)>0$, $(M,g)$ has infinite volume and is $(L^{2}\textnormal{-Liou}_{H})$, then it is $(ESA)$.
\end{enumerate}
\end{proposition}
\begin{proof}
    (1): By the ($ESA$), we have that $H^2(M)=H_0^2(M)$ and the claim follows by Theorem \ref{H_0^p Liou H} below, since, by definition, every $u\in L^2(M)$ so that $-\Delta u=0$ is an element of $H^2(M)$.
    \\
    (2): By the ($L^{2}\textnormal{-Liou}_{H}$) and by the non-finiteness of the volume, the unique $L^2$ harmonic function is the zero function, i.e.
        \begin{align}
            \mathrm{ker}(-\Delta_{2,max})=\{0\}.
        \end{align}
        By Theorem \ref{thm_RS2_X26}, we get the claim.
\end{proof}

\begin{remark}
    We stress that, since $(ESA)$ is equivalent to ($L^{2}\textnormal{-Liou}_{\lambda-H}$) for every $\lambda>0$ (respectively, for every $\lambda>-\lambda_1(M)$ if $M$ has infinite volume), Proposition \ref{ESA-->L2H} also gives the relation between ($L^{2}\textnormal{-Liou}_{\lambda-H}$) and ($L^{2}\textnormal{-Liou}_{H}$).
\end{remark}

\noindent We provide two examples highlighting respectively that, in general, the two properties ($L^{2}\textnormal{-Liou}_{\lambda-H}$) and ($L^{2}\textnormal{-Liou}_{H}$) are not equivalent and that ($L^{2}\textnormal{-Liou}_{H}$) can fail.

\begin{example}[($L^{2}\textnormal{-Liou}_{H}$) but not ($L^{2}\textnormal{-Liou}_{\lambda-H}$)]\label{L2_k not ESA}
    Examples of incomplete manifolds satisfying ($L^{2}\textnormal{-Liou}_{H}$) but not  ($L^{2}\textnormal{-Liou}_{\lambda-H}$) for any $\lambda>0$ are given by $\mathbb{R}^{2}\setminus\{0\}$ and $\mathbb{R}^{3}\setminus\{0\}$ endowed with the standard flat metric. Indeed, by Theorem X.11 in \cite{RS2} we know that the Laplacian on $\mathbb{R}^{n}\setminus\{0\}$ is essentially self-adjoint if and only if $n\ge 4$, so the ($L^{2}\textnormal{-Liou}_{\lambda-H}$) holds only in dimension $n\ge4$. However, $\mathbb{R}^n\setminus\{0\}$ satisfies the ($L^{2}\textnormal{-Liou}_{H}$) property for every $n\ge1$ (actually, they satisfy ($L^{p}\textnormal{-Liou}_{H}$) for every $p\in(1,+\infty)$, see for instance \cite{axler2013harmonic}, Exercise 7, p. 188, Chapter 8).
\end{example}

\begin{example}[Failure of ($L^{2}\textnormal{-Liou}_{H}$)]
    Examples of manifolds for which ($L^{2}\textnormal{-Liou}_{H}$) fails are the hyperbolic disk $\mathbb{H}^{2}\setminus\{0\}$ and the hyperbolic space $\mathbb{H}^{3}\setminus\{0\}$ endowed with the standard hyperbolic metric. Indeed, in both cases the Green function is square-integrable (and actually in $L^{p}$ for every $p\in(1,+\infty)$ when $n=2$). More generally, it was proven in \cite{HMW} that if $(M,g)$ is a complete Riemannian manifold of dimension $2$ or $3$ and with positive bottom of the spectrum, then the incomplete manifold $(M\setminus\{o\},g)$ does not satisfy the $L^{2}$-Liouville property for harmonic functions for any $o\in M$.
\end{example}

\subsubsection*{The case $p\ne2$ }
In the case $p\ne2$, the previous relation can be stated in terms of the core property for the $p$-maximal realization of the Laplacian, as a corollary of the following Liouville property for the domain of the $p$-minimal extension.

\begin{theorem}\label{H_0^p Liou H}
    Let $(M,g)$ be a (possibly incomplete) Riemannian manifold, and let $p\in (1,+\infty)$. Then $(H^{p}_0\textnormal{-Liou}_{H})$ holds.
\end{theorem}
\begin{proof}
    Having fixed $u\in H^{p}_0(M)$ such that $\Delta u=0$, let us consider $\{\phi_k\}_{k\ge1}\subset C^\infty_c(M)$ so that
    \begin{align}
        \phi_k\xrightarrow[]{L^p}u \quad \andd \quad \Delta\phi_k \xrightarrow[]{L^p}\Delta u=0.
    \end{align}
    Up to subsequences, we can also assume that $\phi_{k}$ converges pointwise a.e. to $u$. Now, fix $\varepsilon>0$ and use integration by parts to obtain
    \begin{align}\label{App3_L^p_l-H_implies_L^p_H}
        -\int_M \Delta\phi_k\ (\phi_k^2+\varepsilon)^{\frac{p-2}{2}}\phi_k \dvol &= \int_M  g\left(\nabla\phi_k, \nabla\left[(\phi_k^2+\varepsilon)^{\frac{p-2}{2}}\phi_k\right]\right) \dvol\\
        &=\int_M |\nabla \phi_k|^2 \left[(p-2)(\phi_k^2+\varepsilon)^{\frac{p-4}{2}} \phi_k^2+(\phi_k^2+\varepsilon)^{\frac{p-2}{2}}\right] \dvol\\
        &= \int_M |\nabla \phi_k|^2 (\phi_k^2+\varepsilon)^{\frac{p-4}{2}} \left[(p-2)\phi_k^2 + (\phi_k^2+\varepsilon)\right] \dvol\\
        &\ge (p-1)\int_M |\nabla \phi_k|^2 (\phi_k^2+\varepsilon)^{\frac{p-4}{2}} \phi_k^2 \dvol \\
        &=\frac{4(p-1)}{p^2} \int_M \left|\nabla \left((\phi_{k}^{2}+\varepsilon)^\frac{p}{4}\right)\right|^2 \dvol,
    \end{align}
    hence
    \begin{align}\label{stima energia wk,eps}
        \int_M \left|\nabla w_{k,\varepsilon}\right|^2 \dvol &\le C_{p}\int_M (-\Delta \phi_k) (\phi_{k}^{2}+\varepsilon)^{\frac{p-2}{2}} \phi_k \dvol \\
        &\le C_{p}\int_M \vert\Delta \phi_k\vert (\phi_{k}^{2}+\varepsilon)^{\frac{p-1}{2}} \dvol,
    \end{align}
    where $C_{p}:=\frac{p^2}{4(p-1)}>0$ and $w_{k,\varepsilon}:=(\phi_{k}^{2}+\varepsilon)^\frac{p}{4}\in C^{\infty}(M)$. From the dominated convergence theorem and the smoothness of $\phi_{k}$, the sequence $\{w_{k,\varepsilon}\}_{\varepsilon>0}$ converges to $\vert\phi_{k}\vert^{\frac{p}{2}}$ in $L^{2}_{loc}(M)$ as $\varepsilon\rightarrow 0^+$ for every fixed $k\ge1$. Moreover, from the lower semicontinuity of the energy and the monotone convergence theorem, we deduce that the weak gradient of $\vert\phi_{k}\vert^{\frac{p}{2}}$ exists and
     \begin{align}\label{stima uniforme energia w_k}
        \int_M \left|\nabla \vert\phi_{k}\vert^{\frac{p}{2}}\right|^2 \dvol 
        &\le \liminf_{\varepsilon\longrightarrow 0^{+}}\int_M \vert \nabla w_{k,\varepsilon}\vert^{2}\dvol  \\
        &\le C_{p}\liminf_{\varepsilon\longrightarrow 0^{+}}\int_M \vert\Delta \phi_k\vert(\phi_k^{2}+\varepsilon)^{\frac{p-1}{2}}\dvol \\
        &\le C_{p}\int_M \vert\Delta \phi_k\vert\vert\phi_k\vert^{p-1}\dvol \\
        &\le C_{p}\Vert \Delta \phi_k\Vert_{L^{p}(M)} \Vert\phi_{k}\Vert_{L^{p}(M)}^{\frac{p}{p'}}.
    \end{align}
   Note that the right-hand side of the last inequality is uniformly bounded in $k$ by the convergence in $H^{p}(M)$ of the $\phi_{k}$'s, hence  $\nabla \vert\phi_{k}\vert^{\frac{p}{2}}\in L^{2}(M)$ for every $k\ge1$. Moreover, $\vert\phi_{k}\vert^{\frac{p}{2}}\in L^{2}(M)$ since every $\vert\phi_{k}\vert^{\frac{p}{2}}$ is continuous and compactly supported, and thus $\vert\phi_{k}\vert^{\frac{p}{2}}\in W^{1,2}(M)$. Next step is to prove that $\vert\phi_{k}\vert^{\frac{p}{2}}$ converges in $W^{1,2}_{loc}(M)$ to $\vert u\vert^{\frac{p}{2}}$: to do this, first note that there exists a constant $D_{p}>0$ such that
   \begin{equation*}
       \vert\vert\phi_{k}\vert^{\frac{p}{2}}-\vert u\vert^{\frac{p}{2}}\vert^{2}\le D_{p}\vert \vert\phi_{k}\vert - \vert u\vert\vert^2(\vert\phi_{k}\vert+\vert u\vert)^{p-2}\le D_{p}\vert\phi_{k} - u\vert^2(\vert\phi_{k}\vert+\vert u\vert)^{p-2}
   \end{equation*}
   if $p\in[2,+\infty)$, and 
   \begin{equation*}
       \vert\vert\phi_{k}\vert^{\frac{p}{2}}-\vert u\vert^{\frac{p}{2}}\vert^{2}\le \vert \vert\phi_{k}\vert - \vert u\vert\vert^p
       \le \vert\phi_{k} - u\vert^p
   \end{equation*} 
   if $p\in(1,2)$. In the first case, using the convergence in $L^p$ of the $\phi_{k}$, and applying Holder and Minkowski inequalities, we find
   \begin{equation*}
       \begin{split}
           \Vert \vert\phi_{k}\vert^{\frac{p}{2}}-\vert u\vert^{\frac{p}{2}} \Vert_{L^2}^2 &\le D_{p} \int_{M}\vert\phi_{k} - u\vert^2(\vert\phi_{k}\vert+\vert u\vert)^{p-2} \dvol \\
           & \le D_{p}\Vert \phi_{k} - u \Vert_{L^p}^{2}\Vert \vert\phi_{k}\vert + \vert u\vert \Vert_{L^p}^{p-2} \\
           & \le D_{p}\Vert \phi_{k} - u \Vert_{L^p}^{2}(\Vert\phi_{k}\Vert_{L^p} + \Vert u \Vert_{L^p})^{p-2} \\
           &\le D'_{p}\Vert \phi_{k} - u \Vert_{L^p}^{2},
       \end{split}
   \end{equation*}
   where $D'_{p}$ is a positive constant. In the second case, we immediately get  
   \begin{equation*}
        \Vert \vert\phi_{k}\vert^{\frac{p}{2}}-\vert u\vert^{\frac{p}{2}} \Vert_{L^2}^2\le \Vert \phi_{k} - u \Vert_{L^p}^{2},
   \end{equation*}
   so there exists a positive constant $D'_{p}$ such that
   \begin{equation*}
        \Vert \vert\phi_{k}\vert^{\frac{p}{2}}-\vert u\vert^{\frac{p}{2}} \Vert_{L^2}^2\le D_{p}'\Vert \phi_{k} - u \Vert_{L^p}^{2}
   \end{equation*}
   for every $k\ge1$. Since $\phi_k\xrightarrow[]{H^p}u$, it follows that $\vert\phi_k\vert^{\frac{p}{2}}$ converges to $\vert u\vert^{\frac{p}{2}}$ in $L^{2}(M)$. Moreover, the semicontinuity of the energy implies
   \begin{align}
        \int_{K} \vert \nabla \vert u\vert^{\frac{p}{2}}\vert^{2} \dvol &\le \liminf_{k\rightarrow+\infty}\int_{K} \left|\nabla \vert\phi_{k}\vert^{\frac{p}{2}}\right|^2 \dvol \\
        &\le C_{p}\liminf_{k\rightarrow+\infty}\Vert \Delta \phi_k\Vert_{L^{p}(M)} \Vert\phi_{k}\Vert_{L^{p}(M)}^{\frac{p}{p'}}\\
        &=0.
    \end{align}
    It follows that $\vert \nabla \vert u\vert^{\frac{p}{2}}\vert^{2}=0$ on $M$, which implies that $u$ is constant.
\end{proof}

\noindent As a direct consequence of previous theorem, we obtain

\begin{corollary}\label{p-core -- Lp, Lp' Liou H}
    Let $(M,g)$ be a (possibly incomplete) Riemannian manifold, and let $p\in (1,+\infty)$. If $C^{\infty}_{c}(M)$ is a core for $-\Delta_{p,max}$, then $(L^{p}\textnormal{-Liou}_{H})$ and $(L^{p'}\textnormal{-Liou}_{H})$ hold.
\end{corollary}
\begin{proof}
    If $C^\infty_c(M)$ is a core for $-\Delta_{p,max}$, then $H^p(M)=H^p_0(M)$ and the claim follows by Theorem \ref{H_0^p Liou H}, since every harmonic $u\in L^p(M)$ is an element of $H^p(M)$. The ($L^{p'}\textnormal{-Liou}_{H}$) follows in the same way by duality of the core property, see Proposition \ref{prop_semigroups equivalent cores}.
\end{proof}

By Proposition \ref{prop_semigroups equivalent cores}, it follows that the previous corollary can be reformulated, in an equivalent way, by replacing the core property with the fact that both $-\Delta_{p,min}$ and $-\Delta_{p',min}$ generate a contraction semigroup. However, something stronger can be deduced, as the next theorem states: indeed, we can assert the validity of the ($L^{p}\textnormal{-Liou}_{H}$) requiring only the contractivity of the semigroup generated by $-\Delta_{p',min}$ (or, equivalently, the ($L^{p}\textnormal{-Liou}_{\lambda-H}$) for any $\lambda>0$).

\begin{theorem}\label{Thm_L^p_lambda-H implies L^p_H}
    Let $(M,g)$ be a (possibly incomplete) Riemannian manifold and $p\in (1,+\infty)$. If $(M,g)$ is $(L^{p}\textnormal{-Liou}_{\lambda-H})$ for some $\lambda>0$, then it is $(L^{p}\textnormal{-Liou}_{H})$. If the volume of $M$ is infinite, the same conclusion holds for $\lambda>-C_p(M)$.
\end{theorem}
\begin{proof}
    By Proposition \ref{Lp lambda armomiche e semigruppo}, $-\Delta_{p',min}$ generates a contraction semigroup in $L^{p'}(M)$. Hence, by Remark \ref{Rmk_adjoint norms},
    \begin{align}
        -\Delta_{p,max}=\left(-\Delta_{p',min}\right)^* \quad \textnormal{generates a contraction semigroup on } L^p(M).
    \end{align}
    Let us denote by $P^{(p'),min}_t$ and $P^{(p),max}_t$ the one-parameter (contraction) semigroups generated by $-\Delta_{p',min}$ and $-\Delta_{p,max}$ respectively. By Corollary 10.6, pag. 41 in \cite{pazy2012semigroups}, we have
    \begin{align}\label{App1_Thm_L^p_{lambda-H} implies L^p_H}
        (P^{(p'),min}_t)^*=P^{(p),max}_t.
    \end{align}
    Now fix $u\in L^p(M)$ so that $\Delta u=0$. Let us observe that by \cite[Lemma 1.3]{EN00} one has
    \begin{align}
        \pder{}{t}P_t^{(p),max}u=\Delta P_t^{(p),max}u=P_t^{(p),max}\Delta u=0
    \end{align}
    implying
    \begin{align}\label{App2_Thm_L^p_{lambda-H} implies L^p_H}
        P_t^{(p),max}u=u \quad \forall t>0.
    \end{align}
    Hence, denoting by $v=|u|^{p-2}u$, exploiting \eqref{App1_Thm_L^p_{lambda-H} implies L^p_H}, \eqref{App2_Thm_L^p_{lambda-H} implies L^p_H}, the H\"older inequality and the contractivity of $P_t^{(p'),min}$, one has
    \begin{align}
        \norm{u}{L^p}^p &=<v,u>\\
        &=<v,P_t^{(p),max}u>\\
        &=<P_t^{(p'),min}v,u>\\
        &\leq \norm{P_t^{(p'),min}v}{L^{p'}} \norm{u}{L^p}\\
        &\leq \norm{P_t^{(p'),min}}{\mathcal{L}(L^{p'})} \norm{v}{L^{p'}} \norm{u}{L^p}\\
        &\leq \norm{v}{L^{p'}} \norm{u}{L^p}\\
        &=\norm{u}{L^p}^p
    \end{align}
    for every $t>0$. In particular, all the above inequalities are, in fact, equalities, implying that either $u\equiv 0$ or all the following statements hold:
    \begin{itemize}
        \item $\norm{P_t^{(p'),min}}{\mathcal{L}(L^{p'})}=1$ and $\norm{P_t^{(p'),min}v}{L^{p'}}= \norm{v}{L^{p'}}$;
        \item  by the equality in the H\"older inequality, $|P_t^{(p'),min}v|^{p'}$ and $|u|^p$ have to be proportional: for every $t>0$ there exists a constant $c_t$ so that
    \begin{align}
        \vert P_t^{(p'),min}v\vert=c_t |u|^\frac{p}{p'}=c_t |v| \quad \textnormal{a.e.},
    \end{align}
    hence $c_t=1$ by the previous point. So
    \begin{align}
        |P_t^{(p'),min}v|=|v| \quad \textnormal{a.e.}
    \end{align}
    Moreover, since $uv=|u|^p\geq 0$ and $\vert v\vert=|P_t^{(p'),min}v|$, we get $uv\geq uP_t^{(p'),min}v$ almost everywhere, i.e.
    \begin{align}
        \left[v-P_t^{(p'),min}v\right] u \geq 0 \quad \textnormal{a.e.},
    \end{align}
    and hence, since
    \begin{align}
        \int_M \left[v-P_t^{(p'),min}v\right] u \dvol =0,
    \end{align}
    it follows
    \begin{align}
        P_t^{(p'),min}v=v \quad \textnormal{a.e.}
    \end{align}
    \end{itemize}
    As a consequence, by the definition of the generator of a semigroup, the following limit exists (implying that $v$ is in the domain of the generator, i.e. in $H^{p'}_0(M)$) and it equals
    \begin{align}
        (-\Delta_{p',min})v=\lim_{t\to 0} \frac{v-P_t^{(p'),min}v}{t}=0,
    \end{align}
    hence getting that $v\in H^{p'}_0(M)$ is harmonic. By Theorem \ref{H_0^p Liou H}, we get $v\equiv const.$ and hence $u\equiv const$.
\end{proof}

Combining Theorem \ref{Thm_L^p_lambda-H implies L^p_H} with Remark \ref{rmk_inv_MasaHW_p_ne_2}, we get the following version of \cite[Corollary 4]{HMW} for a general $p\in(1,+\infty)$.
\begin{corollary}
    Let $(M,g)$ be a (possibly incomplete) Riemannian manifold.
    \begin{enumerate}
        \item If $(L^{p}$-\textnormal{Liou}$_{\lambda-H})$ holds (equivalently, $-\Delta_{p',min}$ generates a contraction semigroup in $L^{p'}(M)$), then $(L^{p}$-\textnormal{Liou}$_{H})$ holds.
        \item If $M$ has infinite volume, positive bottom and satisfies $(L^{p}$-\textnormal{Liou}$_{H})$, then $(L^{p}$-\textnormal{Liou}$_{\lambda-H})$ holds.
    \end{enumerate}
\end{corollary}

\noindent We conclude by providing examples that highlight that the ($L^{p}\textnormal{-Liou}_{H}$) property is strictly weaker than the core property and the contraction semigroup property for the Laplacian.
\begin{example}[($L^{p}\textnormal{-Liou}_{H}$) without the core property]\label{contract but not core}
        As we saw in Example \ref{L2_k not ESA}, ($L^{p}\textnormal{-Liou}_{H}$) holds on $\mathbb{R}^{3}\setminus\{0\}$ for every $p\in(1,+\infty)$. Moreover, by \cite{PVV} we know that ($L^{p}\textnormal{-Liou}_{SH}$) holds for $p\ge3$. For $1<p<3$ we have that ($L^{p}\textnormal{-Liou}_{\lambda-H}$) fails: indeed, for every $\lambda>0$ let us consider the radial function $u(x)=\frac{e^{-\sqrt{\lambda}\vert x\vert}}{\vert x\vert}$, then $u\in L^{p}(\mathbb{R}^{3}\setminus\{0\})$, $u$ is non constant and $\Delta u = \lambda u$. In particular, $-\Delta_{p,min}$ generates a contraction semigroup for $p\le \frac{3}{2}$ but not for $p>\frac{3}{2}$, and $C^{\infty}_{c}(M)$ is not a core for the Laplacian for any $p\in(1,+\infty)$ by Corollary \ref{cor_core_equivalence_harmonic_liouville}. This implies both that the $L^{p}$-Liouville property for harmonic functions is strictly weaker than the core property, and that the property of $-\Delta_{p,min}$ being the generator of a contraction semigroup is strictly weaker than the core property.
\end{example}

\begin{example}[($L^{p}\textnormal{-Liou}_{H}$) without the contraction semigroup property]\label{Liou_H_no_semigr_R2}
    As we saw in Example \ref{L2_k not ESA}, ($L^{p}\textnormal{-Liou}_{H}$) holds on $\mathbb{R}^{2}\setminus\{0\}$ for every $p\in(1,+\infty)$. Moreover, ($L^{p}\textnormal{-Liou}_{\lambda-H}$) fails for every $p\in(1,+\infty)$: indeed, for every $\lambda>0$ let us consider the modified Bessel function of the second kind $u(x)=K_{0}(\sqrt{\lambda}\vert x\vert)$. Standard asymptotic estimates yield 
    \begin{align*}
        &K_{0}(\sqrt{\lambda}\vert x\vert)\sim -\log(\sqrt{\lambda}\vert x\vert) \quad\quad\text{as $\vert x\vert\rightarrow0^{+}$,} \\
        &K_{0}(\sqrt{\lambda}\vert x\vert)\sim \sqrt{\frac{\pi}{2\sqrt{\lambda}\vert x\vert}}e^{-\sqrt{\lambda}\vert x\vert} \quad\text{as $\vert x\vert\rightarrow+\infty$}
    \end{align*} 
    (see for instance \cite{abramowitz1948handbook}, pag. 375 and 378), hence $u\in L^{p}(\mathbb{R}^{2}\setminus\{0\})$ for every $p\in(1,+\infty)$. Moreover, $u$ is a non-constant, radial solution to $\Delta u = \lambda u$. This implies that ($L^{p}\textnormal{-Liou}_{\lambda-H}$) fails, and so $-\Delta_{p,min}$ does not generate a contraction semigroup for any $p\in(1,+\infty)$. In particular, this means that this property is not necessary for the validity of ($L^{p}\textnormal{-Liou}_{H}$).
\end{example}

\subsection{Relation between ($L^{2}\textnormal{-Liou}_{SH}$), ($L^{2}\textnormal{-Liou}_{\lambda-SH}$) and ($L^2$-PP) - Part II}\ \\
Thanks to Theorem \ref{Thm_L^p_lambda-H implies L^p_H} and to the properties recalled in Section \ref{Sec_basics}, we are ready to show that the $L^2$-Liouville property for $\lambda$-subharmonic functions is, in fact, equivalent to the $L^2$-Liouville property for subharmonic functions, reversing the implication proved in Corollary \ref{cor_l_sh_l_lambda_sh}.

\begin{proposition}\label{Prop_L^p-Liou_lambda-sh implies L^p-Liou_SH}
Let $(M,g)$ be a (possibly incomplete) Riemannian manifold, and let $p\in(1,+\infty)$. If $M$ satisfies ($L^{p}\textnormal{-Liou}_{\lambda-SH}$) for some $\lambda>0$, then ($L^{p}\textnormal{-Liou}_{SH}$) holds. Moreover, if the volume of $M$ is infinite, the same conclusion holds for every $\lambda\in \rr{}$.
\end{proposition}

\begin{proof}
Preliminarily, recall that the assumption implies $(L^p-\text{PP})$ (by Lemma \ref{lem_monotonicity_L_l-SH}, if $\lambda\leq 1$, or by Proposition \ref{lambda tra 0 e 1}, if $\lambda>1$), and hence that ($L^{p}\textnormal{-Liou}_{H}$) holds.
Consider a function $u$ such that
\begin{equation*}
    \begin{cases}
        \Delta u\geq0 & \text{in $\mathcal{D}'(M)$,}\\
        0\leq u\in L^p(M),
    \end{cases}
\end{equation*}
we want to prove that $u$ is constant. Let $w\in H^{p}(M)$ be the non-negative solution to the equation 
\begin{equation*}
    (-\Delta +1)w = u
\end{equation*}
satisfying the estimate
\begin{equation*}
    \Vert w \Vert_{L^{p}(M)}\le \Vert u \Vert_{L^{p}(M)}.
\end{equation*}
The existence of this function follows by Proposition \ref{solution with non-neg Lp function}. Let us introduce the function $v\in L^{p}(M)$ defined as $v:= w-u$, and note that
\begin{equation*}
    \begin{split}
       (-\Delta+1)(w-u)&=u - (-\Delta u+u) \\
       &=\Delta u\\
       &\geq0
    \end{split}
\end{equation*}
in the sense of distributions. By the $L^p$-positivity preserving property, see Proposition \ref{prop_L_l-SH_and_L^2-PP}, we obtain $v\ge0$, hence
\begin{equation*}
    0\le u\le w.
\end{equation*}
In particular, we deduce the estimate
\begin{equation*}
    \Vert u\Vert_{L^{p}(M)}\leq \Vert w\Vert_{L^{p}(M)},
\end{equation*}
and since $0\leq u\leq w$, it follows that
\begin{equation*}
    w=u
\qquad \text{a.e. on }M.
\end{equation*}
Consequently, 
\begin{equation}
    (-\Delta+1)u=u \ \Longrightarrow \ \Delta u=0,
\end{equation}
hence $u\in L^p(M)$ is harmonic on $M$. The constancy of $u$ is then a consequence of ($L^p$-Liou$_H$).
\end{proof}
\noindent Combining Corollary \ref{cor_equivalence_L2PP_Liouville} with the previous result, we get the following
\begin{corollary}\label{cor_equivalence_L2SH_Liouville_L2PP}
    Let $(M,g)$ be a (possibly incomplete) Riemannian manifold,  let $p\in (1,+\infty)$ and 
    \begin{equation*}
        C_{p}(M):=\frac{4}{pp'}\lambda_{1}(M).
    \end{equation*}
    Then, the following are equivalent:
    \begin{enumerate}
        \item $(L^{p}$\textnormal{-PP}$_\lambda)$ holds for some (equivalently, every) $\lambda>0$;
        \item $(L^{p}\textnormal{-Liou}_{\lambda-SH})$ holds for some (equivalently, every) $\lambda>0$;
        \item $(L^{p}\textnormal{-Liou}_{SH})$ holds.
    \end{enumerate}
    If the volume of $M$ is infinite, the same conclusion holds if we replace $\lambda>0$ with $\lambda>-C_p(M)$.
\end{corollary}


\subsection{Monotonicity properties}\label{sec_monotonicity}\ \\
In this subsection, we study a monotonicity property for ($L^{p}\textnormal{-Liou}_{\lambda-SH}$) with respect to $p$, establishing a hierarchy depending on $p$. We begin with the following
\begin{proposition}[Monotonicity of ($L^{p}\textnormal{-Liou}_{\lambda-SH}$)]\label{Prop_Monotonicity}
    Let $(M,g)$ be a (possibly incomplete) Riemannian manifold and fix $\lambda\geq 0$. If $M$ satisfies the $(L^{p}\textnormal{-Liou}_{\lambda-SH})$ property for $p\in[1,+\infty)$, then it satisfies the $(L^{q}\textnormal{-Liou}_{\lambda-SH})$ property for every $q\in[p,+\infty)$. 
\end{proposition}
\begin{proof}
    Having fixed $\lambda\geq 0$, let $u\in L^{1}_{loc}(M)$ satisfy
    \begin{equation*}
        \begin{cases}
            \Delta u\ge \lambda u \quad \textnormal{\text{in $\mathcal{D}'(M)$}}, \\
            u\ge0, \\
            u\in L^{q}(M),
        \end{cases}
    \end{equation*}
    and set $v:=u^{\alpha}$, where $\alpha:=\frac{q}{p}\ge1$. Clearly, the function $v$ satisfies $v\ge0$ and $v\in L^{p}(M)$, so it remains to prove that $v$ is a $\lambda$-subharmonic distribution, i.e.
    \begin{equation*}
        \int_{M}v(\Delta\varphi-\lambda\varphi)\ \mathrm{dv} \ge 0 \quad\quad \forall \varphi\in C^{\infty}_{c}(M),\ \varphi\ge0.
    \end{equation*}
    Fix $0\le\varphi\in C_{c}^{\infty}(M)$, let $K$ be the support of $\varphi$, and let $\Omega\subset M$ be an open, precompact set with smooth boundary containing $K$. By Theorem \ref{monotone approx thm}, there exists a sequence of positive functions $\{u_{k}\}_{k\ge1}\subset C^{\infty}(\Omega)$ such that
    \begin{enumerate}
        \item $0\le u \le u_{k+1} \le u_{k}\le \dots$ for every $k\ge1$;
        \item $u_{k}(x)\longrightarrow u(x)$ as $k\rightarrow+\infty$ for almost every $x\in \Omega$;
        \item $\Delta u_{k} \ge \lambda u_{k}$ in $\Omega$;
        \item $\Vert u_{k}-u\Vert_{L^{1}(\Omega)}\to 0$ as $k\to\infty$.
    \end{enumerate}
    If we set $v_{k}:=u_{k}^{\alpha}$, it is immediate to see that $0\le v\le v_{k+1} \le v_{k}\le \dots$ for every $k\ge1$, and that $v_{k}(x)\longrightarrow v(x)$ as $k\rightarrow+\infty$ for almost every $x\in \Omega$. Moreover,
    \begin{equation*}
        \begin{split}
            \Delta v_{k}-\lambda v_{k}&= \alpha u_{k}^{\alpha-1}\Delta u_{k} +\alpha(\alpha-1)u_{k}^{\alpha-2}\vert \nabla u_{k}\vert^{2} -\lambda u_{k}^{\alpha} \\
            &\ge \alpha\lambda u_{k}^{\alpha} +\alpha(\alpha-1)u_{k}^{\alpha-2}\vert \nabla u_{k}\vert^{2} -\lambda u_{k}^{\alpha} \\
            &\ge (\alpha-1) (\lambda u_{k}^{\alpha} + \alpha u_{k}^{\alpha-2}\vert \nabla u_{k}\vert^{2})\\
            &\ge 0
        \end{split}
    \end{equation*}
    on $\Omega$, so every $v_{k}$ is a $\lambda$-subharmonic function. It remains to prove the $L^{1}$-convergence of $v_{k}$ to $v$ on $\Omega$. From property (1) and (2) we have 
    \begin{equation*}
        0\le v(x) \le \dots \le v_{k}(x)\le \dots \le v_{1}(x)\le \Vert v_{1} \Vert_{L^{\infty}(\Omega)}<+\infty
    \end{equation*}
    for every $x\in \Omega$ and $k\ge1$. So, by the dominated convergence theorem we have
    \begin{equation*}
        \lim_{k\rightarrow+\infty}\Vert v_{k}-v\Vert_{L^{1}(K)} = \lim_{k\rightarrow+\infty}\int_{K} \vert u_{k}^{\alpha}-u^{\alpha}\vert =0.
    \end{equation*}
    Using the properties of $\{ v_{k} \}_{k\ge1}$ and integration by parts, we deduce that
    \begin{equation*}
        \int_{M}v(\Delta \varphi-\lambda\varphi) =\lim_{k\rightarrow+\infty}\int_{M}v_{k}(\Delta \varphi-\lambda\varphi) =\lim_{k\rightarrow+\infty}\int_{M}\varphi(\Delta v_{k}-\lambda v_{k}) \ge 0. 
    \end{equation*}
    Since $0\le\varphi\in C^{\infty}_{c}(M)$ is arbitrary, we get that $v$ is a distributional solution to $\Delta v\ge \lambda v$. Summing up, we have that 
    \begin{equation*}
        \begin{cases}
            \Delta v\ge \lambda v \quad \inn \mathcal{D}'(M), \\
            v\ge0, \\
            v\in L^{p}(M),
        \end{cases}
    \end{equation*}
    hence $v\equiv const$ by the ($L^{p}\textnormal{-Liou}_{\lambda-SH}$), and so $u\equiv const$.
\end{proof}
\begin{remark}\label{Rmk_Monotonicity L^p-PP}
    We stress that since $(L^{p}$-PP) is equivalent to ($L^{p}\textnormal{-Liou}_{\lambda-SH}$) for any $\lambda\ge0$ and $p\in(1,+\infty)$, then  also ($L^{p}$-PP) is monotone in $p$.
\end{remark}
\begin{remark}
    The case $p=+\infty$ cannot be included: counterexamples are given by complete, stochastically incomplete Riemannian manifolds.
\end{remark}


\section{$L^{p}$-Liouville theorem for subharmonic distributions on manifolds with a small closed set removed}\label{Sec_removed sets}
In this section, we address the problem of the validity of the $L^p$-Liouville property for subharmonic distributions on incomplete manifolds obtained by removing a closed set from a complete Riemannian manifold of dimension $n$, with a suitable control on the Hausdorff dimension. 

The starting point is the paper \cite{PVV}, where the authors solved the BMS conjecture by proving the following $L^p$-Liouville theorem for subharmonic distributions on possibly incomplete Riemannian manifolds.
\begin{theorem}[Theorem 5.1 in \cite{PVV}]\label{PVV Lp con fini}
    Let $(M,g)$ be a (possibly incomplete) non-compact Riemannian manifold with a finite number of ends $E_{1},\dots,E_{N}$. If each end $E_{j}$ is $q_j$-parabolic for some $q_{j}\in(2p',+\infty]$, then $(L^{p}\textnormal{-Liou}_{SH})$ holds.
\end{theorem}
\noindent In particular, since the completeness of $M$ is equivalent to the $\infty$-parabolicity of the manifold, or equivalently of every end (see \cite{PigolaSetti2014global}), they got the following 
\begin{theorem}[Corollary 5.2 in \cite{PVV}]
    Let $(M,g)$ be a complete, non-compact Riemannian manifold. Then $(L^{p}\textnormal{-Liou}_{SH})$ holds for every $p\in(1,+\infty)$.
\end{theorem}
\begin{remark}
    The cases $p=1$ and $p=\infty$ cannot be included: indeed, $(L^{\infty}$-PP) fails on geodesically complete but stochastically incomplete Riemannian manifolds, while counterexamples to $(L^{1}$-PP) can be constructed on some complete manifolds with negative Ricci curvature that explodes to $-\infty$ more than quadratically. For both results we refer to \cite{bisterzo2023positivity}. Moreover, for a result ensuring the validity of $(L^{1}$-PP) on complete manifolds we refer to \cite[Theorem II]{MV24}, where a lower bound on the Ricci curvature is required.
\end{remark}

\noindent From classical relations between capacity of closed sets and Hausdorff dimension (see for instance \cite{troyanov1999parabolicity}), Theorem \ref{PVV Lp con fini} also implies the following 
\begin{theorem}[Corollary 5.7 and Proposition 5.8 in \cite{PVV}]\label{PVV positive res on N-K}
    Let $(N,h)$ be a complete, non-compact Riemannian manifold of dimension $n$, and consider the open manifold $M=N\setminus K$. Let $K\subseteq N$ be a compact subset and fix $p\in(1,+\infty)$. Assume one of the following controls on the size of $K$:
    \begin{itemize}
        \item $\mathrm{dim}_{\mathcal{H}}(K)<n-2p'$, or
        \item $\vert B_{r}(K)\vert \le Cr^{2p'}$ for some $C>0$, where $B_{r}(K):=\{x\in M: d_{g}(x,K)\le r\}$ is the tubular neighborhood of radius $r$ of $K$.
    \end{itemize} 
    Then $(L^{p}\textnormal{-Liou}_{SH})$ holds on $M$.
\end{theorem}
\noindent Note that the dimensional threshold is sharp, because of the following
\begin{theorem}[Proposition 6.2 in \cite{PVV}]\label{thm:negPVV}
    Let $(N,h)$ be a complete, non-compact Riemannian manifold of dimension $n$, and consider the open manifold $M=N\setminus K$. Let $K\subseteq N$ be a closed subset and fix $p\in(1,+\infty)$. If $\mathrm{dim}_{\mathcal{H}}(K)>n-2p'$, then $(L^{p}\textnormal{-Liou}_{SH})$ and $(L^{p}$-$\textnormal{PP})$ do not hold on $M$.
\end{theorem}

\begin{remark}\label{not_n=2}
    Note that when $n=2$, if we remove a point $p$ from a complete surface, all the Liouville properties except for ($L^{p}$-Liou$_{H}$) fail in general. Indeed, in this case $\text{dim}_\mathcal{H}(K)\geq 0>n-2p'$ for every $p\in(1,+\infty)$, and Theorem \ref{thm:negPVV} implies that ($L^{p}$-Liou$_{SH}$) does not hold. On the other hand, by \cite{masammune1999essential} the Laplacian is not essentially self adjoint if we remove a closed submanifold of codimension less than or equal to 3 (see also \cite{hinz2023removable} and Example \ref{Liou_H_no_semigr_R2} for the case $p\ne2$).
\end{remark}

\noindent The latter result naturally suggests the question whether only complete manifolds satisfy ($L^{p}\textnormal{-Liou}_{SH}$) for all $p$; see \cite[Section 1.4]{PVV}. The following example shows that it is not the case. 
\begin{example}
    Consider the warped product $(M,g)=((0,\infty)\times \mathbb S^{n-1}, \text{d}t^2 + j^2(t)\mathring{g})$, where $\mathring{g}$ denotes the round metric on $\mathbb{S}^{n-1}$. Choose a positive warping function $j\in C^{\infty}((0,+\infty))$ which decays more than polynomially to $0$ both as $t\to 0^+$ (such as for instance $j(t)=\exp(-1/t^2)$ when $t<1$) and as $t\to\infty$ (so that $M$ has finite volume). 
Let $B^\ast_r(0)$ be the punctured ball centered at the origin, i.e. $B^\ast_r(0)=(0,r)\times \mathbb S^{n-1}\subset M$. Then, for every $\alpha\in\mathbb{R}$ we have 
\begin{equation*}
    \lim_{r\to 0}r^\alpha \operatorname{vol}_g(B^\ast_r(0)) = 0.
\end{equation*} 
Let $\varphi_\varepsilon\in C^\infty (M)$  be a cutoff function such that 
\[\varphi_\varepsilon(t,\theta) = \begin{cases}1 & \text{if $t>\varepsilon$,}\\
0 & \text{if $t<\varepsilon/2$},
\end{cases}
\]
with $\|\nabla\varphi_\varepsilon\|_\infty\le 4/\varepsilon$. Then 
\[
\int_{M} |\nabla \varphi_\varepsilon|^p \dvol_g \le 4^p \varepsilon^{-p}  \operatorname{vol}_g(B^\ast_\varepsilon(0)) \to 0,\ \text{as }\varepsilon\to 0.
\]
In particular, the end of $M$ corresponding to $t=0$ is $p$-parabolic for all $p<\infty$ (but not complete) while the end corresponding to $t=\infty$  is $p$-parabolic for all $p\in(1, \infty]$ as it is complete and has finite volume. An application of Theorem 5.1 in \cite{PVV} implies the ($L^p$-PP) for every $p\in(1,+\infty)$. 
\end{example}

\begin{remark}
    A complementary, potential-theoretic approach to Liouville-type properties was developed by Grigor'yan and collaborators, who related Liouville phenomena to recurrence and non-explosion of Brownian motion, finite Dirichlet integrals, Green formulae, and parabolicity, and more recently to $L^p$-parabolicity and $L^q$-Liouville properties for positive superharmonic functions; see \cite{Grigor1982,Grigor1988,Grigor1999,GM,GrigorPessoaSetti}. These results provide characterizations of Liouville-type properties from a potential-theoretic and probabilistic viewpoint. We have to emphasize, however, that the $L^q$-Liouville property studied in the latter work, \cite{GrigorPessoaSetti}, is formulated in terms of non-negative superharmonic functions, whereas in this paper we focus on the corresponding Liouville property for non-negative subharmonic functions. This distinction is not merely terminological: for finite $q$, the two notions are genuinely different in general and are naturally connected with different potential-theoretic and operator-theoretic properties.
\end{remark}

\subsection{Removing closed sets with locally finite Hausdorff measure}\ \\
The aim of this section is twofold: first, we want to remove the compactness assumption in Theorem \ref{PVV positive res on N-K}, considering possibly unbounded closed sets. Secondly, we want to give a more precise condition in the threshold case $n=2p'$, related to a generalized notion of Hausdorff measure. The argument relies on the construction of suitable cutoff functions, whose existence is guaranteed by the control on the Hausdorff dimension. Moreover, in view of Remark \ref{not_n=2} we will restrict to manifolds of dimension greater than 2.

\subsubsection{Subcritical case} For the case $\mathrm{dim}_{\mathcal{H}}(K)<n-2p'$, the construction of the cutoff function is a consequence of the following
\begin{lemma}[Lemma 3.2 in \cite{harvey1970removable}]\label{Lemma cutoff Hausdorf}
    Let $p\in \left(\frac{n}{n-2},+\infty\right)$ and $K\subseteq \mathbb{R}^{n}$ be a compact set with finite $k$-dimensional Hausdorff measure, where $k:=n-2p'$. Then, for every $\varepsilon>0$ there exists $\varphi_{\varepsilon}\in C^{\infty}_{c}(\mathbb{R}^{n})$ such that:
    \begin{enumerate}
        \item $\varphi_{\varepsilon}\equiv1$ in a neighborhood of $K$;
        \item $0\le\varphi_{\varepsilon}\le1$; 
        \item $\mathrm{supp}\ \varphi_{\varepsilon}\subseteq B_{\varepsilon}(K)$;
        \item $\Vert D^{\alpha}\varphi_{\varepsilon}\Vert_{L^{p'}}\le C_{\alpha}\varepsilon^{2-\vert\alpha\vert}(\mathcal{H}^{k}(K)+\varepsilon)^{\frac{1}{p'}}$ for every multi-index $\alpha \in \mathbb{N}^{n}$ with $\vert\alpha\vert\le2$, where $C_{\alpha}$ is a constant independent of $\varepsilon$.
    \end{enumerate}
\end{lemma}
\begin{remark}
    We stress that, although condition (2) is not explicitly stated in the original paper, an inspection of the proof shows that also this property holds if we choose wisely the first cutoff in the proof of Lemma 3.2 of \cite{harvey1970removable}. 
\end{remark}
\noindent The construction of similar cutoff functions can be done also on Riemannian manifolds, assuming that the compact set $K$ is contained in a coordinate chart, as shown by the following key lemma.

\begin{lemma}\label{Lemma cutoff Hausdorff in carte locali}
    Let $(M,g)$ be a Riemannian manifold of dimension $n\geq 3$ and let $K\subseteq M$ be a compact set which is contained in the domain of a local chart of $M$. Assume that $K$ has finite $k$-dimensional Hausdorff measure for $k:=n-2p'$, where $p\in \left(\frac{n}{n-2},+\infty\right)$. Then, for every $\widetilde{K}\subset M$ compact with $\textnormal{int}(\widetilde{K})\neq \emptyset$ so that $K\Subset \textnormal{int}(\widetilde{K})$ and for every $\varepsilon>0$ small enough there exists $\varphi_\varepsilon\in C^\infty_c(M)$ such that:
    \begin{enumerate}
        \item $\varphi_\varepsilon\equiv 1$ in a neighborhood $K_\varepsilon$ of $K$;
        \item $0\leq \varphi_\varepsilon\leq 1$;
        \item $\mathrm{supp}\ \varphi_{\varepsilon} \subseteq  T_{\varepsilon}\subset \widetilde{K}$, where $T_{\varepsilon}$ are open neighborhoods of $K$ such that $\vert T_{\varepsilon}\vert\longrightarrow0$ as $\varepsilon\rightarrow 0^{+}$;
        \item there exists a positive constant $C$ which does not depend on $\varepsilon$ such that 
        \begin{equation*}
           \Vert \varphi_{\varepsilon}\Vert_{W^{2,p'}(M)}\le C(\mathcal{H}^{k}   (K)+\varepsilon)^{\frac{1}{p'}},
        \end{equation*}
        where $\Vert\cdot\Vert_{W^{2,p'}(M)}:=\Vert\cdot\Vert_{L^{p'}(M)}+\Vert\nabla\cdot\Vert_{L^{p'}(M)}+\Vert\mathrm{Hess}(\cdot)\Vert_{L^{p'}(M)}$.
    \end{enumerate}
\end{lemma}

\begin{proof}
    Let $\widetilde{K}\subset M$ be a compact subset such that $\textnormal{int}(\widetilde{K})\neq \emptyset$ and $K\Subset \textnormal{int}(\widetilde{K})$, and let $(\mathcal{U},\Phi)$ be a local chart for $M$ such that $K\subset\mathcal{U}$. After shrinking $\widetilde{K}$, we can assume $K\Subset \textnormal{int}(\widetilde{K})\subset\widetilde{K}\Subset\mathcal{U}$. Consider the compact sets $F:=\Phi(K)\subseteq\mathbb{R}^{n}$ and $\widetilde{F}:=\Phi(\widetilde{K})$. Having fixed $\varepsilon>0$ small enough, by Lemma \ref{Lemma cutoff Hausdorf} we can find a cutoff function $\psi_{\varepsilon}\in C^{\infty}_{c}(\mathbb{R}^{n})$, $0\le\psi_{\varepsilon}\le1$, such that $\psi_{\varepsilon}\equiv1$ in a neighborhood $V$ of $F$, $\mathrm{supp}\ \psi_{\varepsilon}\subseteq 
    B_{\varepsilon}(F)\subseteq\widetilde{F}\subseteq \Phi(\mathcal{U})$, and satisfying
    \begin{equation*}
        \Vert D^{\alpha}\psi_{\varepsilon}\Vert_{L^{p'}}\le C_{\alpha}\varepsilon^{2-\vert\alpha\vert}(\mathcal{H}^{k}(F)+\varepsilon)^{\frac{1}{p'}}
    \end{equation*}
    for every multi-index $\alpha$ with $\vert\alpha\vert\le2$ and where $C_{\alpha}$ is a constant independent of $\varepsilon$. Define the function
    \begin{equation*}
        \varphi_{\varepsilon}(x):=
        \begin{cases}
            \psi_{\varepsilon}(\Phi(x)) & \text{if $x\in\mathcal{U}$,}\\
            0 & \text{if $x\in M\setminus\mathcal{U}$,}
        \end{cases}
    \end{equation*}
    then $\varphi_{\varepsilon}\in C^{\infty}_{c}(M)$, $0\le\varphi_{\varepsilon}\le1$, $\varphi_{\varepsilon}\equiv1$ in a neighborhood $K_{\varepsilon}:=\Phi^{-1}(V)$ of $K$, and $\mathrm{supp}\ \varphi_{\varepsilon}\subseteq T_{\varepsilon}:=\Phi^{-1}(B_{\varepsilon}(F))\subset\widetilde{K}$. Note that the measure of $T_{\varepsilon}$ goes to zero as $\varepsilon\rightarrow0^{+}$ since the measure of $B_{\varepsilon}(F)$ goes to zero and the Lebesgue measures of $M$ and $\mathbb{R}^{n}$ are comparable. 
    
    Finally, to prove the estimates on the $W^{2,p'}$-norm it is enough to write the integrals in local coordinates and to use the regularity of the metric coefficients and the previous Lemma on the local representation of the function $\varphi_{\varepsilon}$.
\end{proof}

\subsubsection{Critical case}
In the critical case $n=2p'$ the previous construction does not work, so we need to follow a different path to obtain an analogous cutoff function. Due to the lack of information provided by the standard $0$-dimensional Hausdorff measure, we are led to consider a generalized Hausdorff measure which is better suited to capture the different ``sizes" of 0-dimensional subsets. The measure we consider is  $\mathcal{H}^{h_{p'}}$, where
\begin{align}
    h_{p'}(r):=\left[1+\ln_+\left(\frac{1}{r}\right)\right]^{1-{p'}}.
\end{align}
For further details on generalized Hausdorff measures, see \cite[Section 8]{hinz2023removable} or \cite{AH12}.

The technical lemma we are going to prove is the analogue of Lemma \ref{Lemma cutoff Hausdorf} for the measure $\mathcal{H}^{h_{p'}}$.

\begin{lemma}\label{Lem_h_p_Hausdorff_cut-off}
    Let $K\subset \rr{n}$ be a compact set satisfying $\mathcal{H}^{h_{p'}}(K)<+\infty$, where $p'=\frac{n}{2}$. Then, for every $\widetilde{K}\subset \mathbb{R}^n$ compact with $\textnormal{int}(\widetilde{K})\neq \emptyset$ so that $K\Subset \textnormal{int}(\widetilde{K})    $ and for every $\varepsilon>0$ there exists $\varphi_\varepsilon\in C^\infty_c(\rr{n})$ such that:
    \begin{enumerate}
        \item $\varphi_\varepsilon\equiv 1$ in a neighborhood $K_\varepsilon$ of $K$;
        \item $0\leq \varphi_\varepsilon\leq 1$;
        \item $\mathrm{supp}\ \varphi_{\varepsilon} \subseteq  T_{\varepsilon}\subset \widetilde{K}$, where $T_{\varepsilon}$ are open neighborhoods of $K$ such that $\vert T_{\varepsilon}\vert\longrightarrow0$ as $\varepsilon\rightarrow 0^{+}$;
        \item $\norm{\varphi_\varepsilon}{W^{2,p'}(\rr{n})}\leq \varepsilon$.
    \end{enumerate}
\end{lemma}

\noindent In order to prove Lemma \ref{Lem_h_p_Hausdorff_cut-off} we start introducing the following notation. We denote by $G_2$ the $2$-Bessel kernel defined, in terms of the Fourier transform, as
\begin{align}
    G_\alpha=\mathcal{F}^{-1}\left[(1+|\xi|^2)^{-1}\right].
\end{align}
Denoting by $\mathcal{S}$ the Schwartz class, we recall that 
\begin{align}\label{G_a_bijection}
G_2:\mathcal{S}\to \mathcal{S}\quad \textnormal{is a bijection},
\end{align}
where $G_{2}$ acts by convolution on $\mathcal{S}$. For a reference, see \cite[Section 1.1.7]{AH12}. For a fixed compact set $K\subset \rr{N}$ and $q>1$, let us define the following capacities
\begin{align}
    & C_{2,q}'(K):=\inf\left\{\norm{\varphi}{W^{2,{q}}}^{q} : \varphi\in C_c^\infty(\rr{n}),\ \varphi\geq 1\ \onn K \right\}\\
    & C_{G_{2},q}(K):=\inf\left\{\norm{f}{L^{q}}^{q} : 0\leq f \in L^{q}(\rr{n}),\ G_2 *f\geq 1\ \onn K\right\}.
\end{align}
Thanks to results from \cite{AH12}, namely Theorem 1.2.3, the observations at page 20 and at pages 47-48 and Proposition 2.3.13, the capacities defined above are equivalent. In particular, there exist two positive constants $A_1$ and $A_2$, depending only on $n$ and $q$, such that
\begin{align}\label{Eq_Equiva_capacities_0}
    C_{2,q}'(K)\leq A_1 C_{G_{2},q}(K)\leq A_2 C_{2,q}'(K)
\end{align}
for every compact subset $K\subseteq\rr{n}$. Now let us define
\begin{align}
    N_{2,q}'(K):=\inf\{\norm{\varphi}{W^{2,{q}}}^{q} : \varphi \in C^\infty_c(\rr{n}),\ 0\leq \varphi\leq 1,\ \varphi=1\ \textnormal{in a neigh. of } K\}.
\end{align}
Clearly, it holds
\begin{align}\label{C' leq N'}
    C_{2,q}'(K)\leq  N_{2,q}'(K).
\end{align}
Next lemma states that, up to a multiplicative constant, the above inequality can be reversed.

\begin{lemma}\label{equivalence capacities}
    The capacity $N_{2,q}'$ is equivalent to $C_{G_{2},q}$ (and hence to $C'_{2,q}$).
\end{lemma}

\begin{proof}Let us fix $T:\rr{}\to [0,1]$ a smooth function so that
    \begin{align}
        T(s)=\begin{cases}
            0 & \textnormal{if}\ s\leq \frac{1}{2}\\
            1 & \textnormal{if}\ s\geq 1.
        \end{cases}
    \end{align}
For any fixed $0\leq f \in L^q(\rr{n})$ such that $G_2*f\geq 1$ on $K$ and $\varepsilon>0$, we denote
\begin{align}
    F^\varepsilon:=(1+\varepsilon)G_2*f,
\end{align}
which is lower semicontinuous by \cite[Proposition 2.3.2]{AH12}. Hence, the set 
\begin{align}
U_\varepsilon:=\{x\in \rr{n} \mid F^\varepsilon>1\}    
\end{align}
is open and $K\subset U_\varepsilon$. Next, consider
\begin{align}
    T^\varepsilon:=T(F^\varepsilon)
\end{align}
and observe that $0\leq T^\varepsilon \leq 1$, $T^\varepsilon=1$ on $U_\varepsilon$ and, by \cite[Theorem 3.3.3]{AH12}, $T^\varepsilon \in W^{2,q}(\rr{n})$. Moreover, by \cite[Theorem 1.2.3 and Theorem 3.3.3]{AH12} there exists a positive constant $A_1$, depending only on $q$ and $n$, such that
\begin{align}
    \norm{T^\varepsilon}{W^{2,q}}\leq A_1 (1+\varepsilon) \norm{f}{L^q}
\end{align}
Now consider a one-parameter family of compactly supported, smooth, cut-off functions $\{\psi_R\}_{R>0}\subset C^\infty_c(\rr{n})$ such that
\begin{itemize}
    \item for every $R>0$ we have $\psi_R\equiv 1$ on $B_R$ and $\textnormal{supp}(\psi_R)\Subset B_{2R}$;
    \item $\{\psi_R\}_R$ has equibounded gradient and Hessian.
\end{itemize}
Without loss of generality, we assume $R>R_0$, where $R_0>0$ is such that $K\subset B_{R_0}$. We define
\begin{align}
    T^\varepsilon_{R}:=\psi_R T^\varepsilon,
\end{align}
which is compactly supported, in $W^{2,q}(\rr{n})$ and such that
\begin{align}
    \norm{T_R^\varepsilon}{W^{2,q}}\leq A_2 \norm{T^\varepsilon}{W^{2,q}}
\end{align}
for a positive constant $A_2$, depending only on $q$ and $n$.
Let $\{\rho_\delta\}_{\delta>0}$ be the family of Friedrichs mollifiers and denote
\begin{align}
    \varphi_{R,\delta}^\varepsilon:=\rho_\delta*T^\varepsilon_R.
\end{align}
Moreover, by the properties of the mollification and of the Friedrichs mollifiers
\begin{itemize}
    \item $\varphi_{R,\delta}^\varepsilon\in C^\infty_c(\rr{n})$;
    \item $0\leq \varphi_{R,\delta}^\varepsilon\leq 1$;
    \item  there exists a positive constant $A_3$, depending only on $q$ and $n$, such that
    \begin{align}
        \norm{\varphi_{R,\delta}^\varepsilon}{W^{2,q}}\leq A_3 \norm{T^\varepsilon_R}{W^{2,q}}.
    \end{align}
\end{itemize}
Lastly, for $\delta>0$ small enough, we have
\begin{align}
    \varphi_{R,\delta}^\varepsilon\equiv 1 \quad \text{in a neighbourhood of } K.
\end{align}
Thus, $\varphi^\varepsilon_{R,\delta}$ is an admissible competitor in the definition of $N'_{2,q}(K)$. It then follows that for every sufficiently large $R>0$, and for every $\varepsilon>0$ and $\delta>0$ small enough,
\begin{align}
    N'_{2,q}(K) & \leq \norm{\varphi_{R,\delta}^\varepsilon}{W^{2,q}}^{q}\\
    &\leq A_3 \norm{T^\varepsilon_R}{W^{2,q}}^{q}\\
    &\leq A_2 A_3 \norm{T^\varepsilon}{W^{2,q}}^{q}\\
    &\leq A_1 A_2 A_3 (1+\varepsilon) \norm{f}{L^q}^{q}\\
    &=:A_4 (1+\varepsilon) \norm{f}{L^q}^{q}.
\end{align}
By taking the infimum over all admissible $f$ in the definition of $C_{G_{2},q}$, and letting $\varepsilon \to 0$, it then follows that
\begin{align}
    N'_{2,q}(K)\leq A_4 C_{G_{2},q}(K).
\end{align}
By putting together the above inequality with \eqref{Eq_Equiva_capacities_0} and \eqref{C' leq N'}, the claim follows.
\end{proof}

\begin{proof}[Proof of Lemma \ref{Lem_h_p_Hausdorff_cut-off}]
    By \cite[Theorem 5.1.9]{AH12}, $C_{2,p'}'(K)=0$, since $C_{2,p'}'$ is equivalent to the capacity appearing in \cite[Page 20]{AH12} as $C_{2,p'}$. By Lemma \ref{equivalence capacities} it follows that $N_{2,p'}'(K)=0$. Hence, for every $\varepsilon >0$ there exists $\widetilde{\varphi}_\varepsilon \in C^\infty_c(\rr{n})$ such that
    \begin{itemize}
         \item $\widetilde{\varphi}_\varepsilon\equiv 1$ in a neighbourhood $K_\varepsilon$ of $K$;
        \item $0\leq \widetilde{\varphi}_\varepsilon\leq 1$;
        \item $\norm{\widetilde{\varphi}_\varepsilon}{W^{2,p'}}\leq \varepsilon$.
    \end{itemize}
    To conclude, we have to show that for every $\varepsilon>0$ we can construct $\varphi_\varepsilon$ with support in an open neighbourhood $T_\varepsilon$ of $K$ such that $|T_\varepsilon|\to 0$ as $\varepsilon \to 0^+$. To this aim, consider a nested family $\{T_\varepsilon\}_\varepsilon$ of bounded, open neighbourhoods of $K$ so that
    \begin{align}
        \widetilde{K}\supset T_\varepsilon \supset T_\delta \supset K \quad \forall \varepsilon>\delta>0
    \end{align}
    and satisfying
    \begin{align}
        |T_\varepsilon|\to 0 \quad \textnormal{as}\ \varepsilon\to 0^+.
    \end{align}
    Let $\eta_\varepsilon\in C^\infty_c(T_\varepsilon)$ so that
    \begin{align}
        0\leq \eta_\varepsilon \leq 1 \quad \andd \quad \eta_\varepsilon\equiv 1 \ \text{in a neighbourhood of } K
    \end{align}
    and denote
    \begin{align}
        C_\varepsilon:=\norm{\eta_\varepsilon}{W^{2,\infty}}<+\infty.
    \end{align}
    Given $\{\widetilde{\varphi}_\delta\}_\delta$ the minimizing sequence for $N_{2,p'}'(K)$ introduced above, we define
    \begin{align}
        \varphi_{\varepsilon,\delta}:=\eta_\varepsilon \widetilde{\varphi}_\delta.
    \end{align}
    We have that $0\leq \varphi_{\varepsilon,\delta}\leq 1$, $\varphi_{\varepsilon,\delta}=1$ in a neighborhood of $K$ and
    \begin{align}
        \textnormal{supp}\ \varphi_{\varepsilon,\delta}\subseteq \textnormal{supp}\ \eta_\varepsilon \subset T_\varepsilon.
    \end{align}
    Moreover
    \begin{align}
        \norm{\varphi_{\varepsilon,\delta}}{W^{2,p'}} &\leq 3C_\varepsilon \norm{\widetilde{\varphi}_\delta}{W^{2,p'}}\\
        &\leq 3C_\varepsilon \delta.
    \end{align}
    Hence, choosing $\delta=\delta(\varepsilon)>0$ small enough so that $3C_\varepsilon \delta<\varepsilon$, we only have to define
    \begin{align}
        \varphi_\varepsilon:=\varphi_{\varepsilon, \delta(\varepsilon)}.
    \end{align}
    The claim follows.
\end{proof}

\noindent The next lemma is a straightforward consequence of Lemma \ref{Lem_h_p_Hausdorff_cut-off} and of the smoothness of $(M,g)$, and it is the analogue of Lemma \ref{Lemma cutoff Hausdorff in carte locali} for the measure $\mathcal{H}^{h_{p'}}$.

\begin{lemma}\label{Lemma cutoff Hausdorff in carte locali caso h_p}
    Let $(M,g)$ be a Riemannian manifold of dimension $n\geq 3$ and let $K\subseteq M$ be a compact set which is contained in the domain of a local chart of $M$. Assume that $\mathcal{H}^{h_{p'}}(K)$ is finite, where $p'=\frac{n}{2}$. Then, for every $\widetilde{K}\subset M$ compact with $\textnormal{int}(\widetilde{K})\neq \emptyset$ so that $K\Subset \textnormal{int}(\widetilde{K})$ and for every $\varepsilon>0$ small enough there exists $\varphi_\varepsilon\in C^\infty_c(M)$ such that:
    \begin{enumerate}
        \item $\varphi_\varepsilon\equiv 1$ in a neighborhood $K_\varepsilon$ of $K$;
        \item $0\leq \varphi_\varepsilon\leq 1$;
        \item $\mathrm{supp}\ \varphi_{\varepsilon} \subseteq  T_{\varepsilon}\subset \widetilde{K}$, where $T_{\varepsilon}$ are open neighborhoods of $K$ such that $\vert T_{\varepsilon}\vert\longrightarrow0$ as $\varepsilon\rightarrow 0^{+}$;
        \item there exists a positive constant $C$, not depending on $\varepsilon$, such that
        \begin{align}
            \norm{\varphi_\varepsilon}{W^{2,p'}(M)}\leq C\varepsilon.
        \end{align}
    \end{enumerate}
\end{lemma}

\subsection{Liouville property and consequences on Schrodinger operators}\ \\
In the remainder of this section, for every $p\geq \frac{n}{n-2}$ let $\mathfrak{M}_{p}$ be the measure given by 
\begin{equation*}
    \mathfrak{M}_{p}:=\begin{cases}
        \mathcal{H}^{n-2p'} & \text{if $n>2p'$}\\
        \mathcal{H}^{h_{p'}} & \text{if $n=2p'$.}
    \end{cases}
\end{equation*}
Thanks to the existence of good cutoff functions, we can prove the following proposition.

\begin{proposition}\label{Removability of singularity}
    Let $(M,g)$ be a (possibly incomplete) Riemannian manifold of dimension $n\ge3$, $p\in \left[\frac{n}{n-2},+\infty\right)$ and $K\subseteq M$ be a closed set. If
    \begin{equation*}
        \mathfrak{M}_{p}(\widetilde{K})<+\infty \quad\quad \text{for every compact $\widetilde{K}\subseteq K$},
    \end{equation*}
    then
    \begin{equation*}
        \begin{cases}
            \Delta u\ge0 &\text{in $\mathcal{D}'(M\setminus K)$}\\
            0\le u\in L^{p}_{loc}(M)
        \end{cases}
        \Longrightarrow \Delta u\ge0 \quad\text{in $\mathcal{D}'(M)$.}
    \end{equation*}
\end{proposition}
\begin{proof}
    For a fixed $0\le\psi\in C^{\infty}_{c}(M)$, we want to show that 
    \begin{equation*}
        \langle \Delta u,\psi\rangle\ge0.
    \end{equation*}
    We begin by considering the case of $\psi$ having support contained in the domain of a coordinate chart $(\mathcal{U},\Phi)$. Set $\widetilde{K}:=\text{supp}  \,\varphi \cap K$, then $\widetilde{K}$ is a compact set contained in the local chart $\mathcal{U}$ and $\mathfrak{M}_{p}(\widetilde{K})<+\infty$ by hypothesis, so we can apply Lemma \ref{Lemma cutoff Hausdorff in carte locali} or Lemma \ref{Lemma cutoff Hausdorff in carte locali caso h_p}, depending on whether $2p'<n$ or $2p'=n$ respectively, to find for every $\varepsilon\in(0,1]$ a cutoff function $\varphi_{\varepsilon}\in C^{\infty}_{c}(M)$ satisfying the properties of the cited lemmas. 
    
    Consider the distributional identity
    \begin{equation*}
        \langle \Delta u, \psi\rangle = \langle \Delta u, \psi\varphi_{\varepsilon}\rangle + \langle \Delta u, \psi(1-\varphi_{\varepsilon})\rangle.
    \end{equation*}
    Since $\psi(1-\varphi_{\varepsilon})\in C^{\infty}_{c}(M\setminus K)$ is non-negative, we have $\langle \Delta u, \psi(1-\varphi_{\varepsilon})\rangle\ge0$ for every $\varepsilon>0$. It remains to analyze the first term on the right-hand side:
    \begin{equation*}
        \begin{split}
            \vert \langle \Delta u, \psi\varphi_{\varepsilon}\rangle\vert &\le\int_{M}\vert u\vert \Big[\vert \varphi_{\varepsilon}\vert\Vert\Delta\psi\Vert_{L^{\infty}}+\Vert \psi\Vert_{L^{\infty}}\vert\Delta\varphi_{\varepsilon}\vert+2\Vert\nabla\psi\Vert_{L^{\infty}}\vert\nabla \varphi_{\varepsilon}\vert \Big] \dvol \\
            &\le C \int_{T_{\varepsilon}}\vert u\vert \Big[\vert \varphi_{\varepsilon}\vert+\vert\Delta\varphi_{\varepsilon}\vert+\vert\nabla \varphi_{\varepsilon}\vert \Big] \dvol\\
            &\le C \Vert u\Vert_{L^{p}(T_{\varepsilon})}\norm{\varphi_\varepsilon}{W^{2,p'}(M)}.
        \end{split}
    \end{equation*}
    Since $u\in L^{p}_{loc}(M)$, the right-hand side of the inequality tends to zero as $\varepsilon\rightarrow0^{+}$, hence
    \begin{equation*}
        \lim_{\varepsilon\rightarrow0^{+}}\langle \Delta u, \psi\varphi_{\varepsilon}\rangle = 0,
    \end{equation*}
    and then 
    \begin{equation*}
        \langle \Delta u, \psi\rangle = \liminf_{\varepsilon\rightarrow0^{+}} (\langle \Delta u, \psi\varphi_{\varepsilon}\rangle + \langle \Delta u, \psi(1-\varphi_{\varepsilon})\rangle)\ge \lim_{\varepsilon\rightarrow0^{+}} \langle \Delta u, \psi\varphi_{\varepsilon}\rangle = 0.
    \end{equation*}
    To conclude, consider a general $0\le\psi\in C^{\infty}_{c}(M)$, and let $\Omega:=\mathrm{supp}\psi$. Let $\{(\mathcal{U}_i,\Phi_{i})\}_{i=1,\dots,N}$ be an open cover of $\Omega$ made of local charts, let $\Omega_{1}$ be an open neighborhood of $\Omega$ such that $\overline{\Omega}_{1}\subseteq\bigcup_{i=1}^{N}\mathcal{U}_{i}$, and set $\mathcal{V}:=M\setminus\overline{\Omega}_{1}$. Then $\{\mathcal{U}_{i}\}_{i}\cup\{\mathcal{V}\}$ is an open cover of $M$, so we can find a partition of unity $\{\rho_{1},\dots,\rho_{N},\rho\}$ subordinate to this covering. Note that
    \begin{equation*}
        \psi=\sum_{i=1}^{N}\psi\rho_{i}=:\sum_{i=1}^{N}\psi_{i},
    \end{equation*}
    and $\mathrm{supp}\psi_{i}\subseteq \mathcal{U}_{i}$. From the previous discussion, we have
    \begin{equation*}
        \langle \Delta u,\psi\rangle = \sum_{i=1}^{N}\langle\Delta u,\psi_{i}\rangle \ge0,
    \end{equation*}
    and this concludes the proof.
\end{proof}

As a straightforward consequence of Proposition \ref{Removability of singularity} and the monotonicity of ($L^{p}\textnormal{-Liou}_{SH}$) with respect to $p$, we have the following 
\begin{theorem}\label{Liouville subharm on manif with cls rem}
    Let $(N,h)$ be a complete Riemannian manifold of dimension $n\ge3$, $p\in \left[\frac{n}{n-2},+\infty\right)$ and $K\subseteq N$ be a closed set. Consider the open manifold $M=N\setminus K$ and  and assume that   
    \begin{equation*}
        \mathfrak{M}_p(\widetilde{K})<+\infty \quad\quad \text{for every compact $\widetilde{K}\subseteq K$}.
    \end{equation*}
    Then, ($L^{q}\textnormal{-Liou}_{SH}$) holds on $M$ for every $q\in[p,+\infty)$.  
\end{theorem}

\noindent In particular, the above result allows us to deduce important functional properties of Schrödinger operators with potential $0\le V\in L^{p}_{loc}(M)$ on incomplete manifolds.

\begin{theorem}\label{Schrod on manif with cls rem}
    Let $(N,h)$ be a complete Riemannian manifold of dimension $n\ge3$, $p\in \left(1,\frac{n}{2}\right]$ and $K\subseteq N$ a closed set. Consider the open manifold $M:=N\setminus K$ and let $0\le V\in L^{p}_{loc}(M)$. If
    \begin{equation*}
        \mathfrak{M}_{p'}(\widetilde{K})<+\infty \quad\quad \text{for every compact $\widetilde{K}\subseteq K$},
    \end{equation*}
    then $(-\Delta+V)_{q,min}$ generates a contraction semigroup in $L^{q}(M)$ for all $1<q\le p$.
\end{theorem}

As a direct consequence of the above theorem for $q=2$, we obtain the following corollary, which extends Theorem 10.12.(II) from \cite{hinz2023removable} to the case of general potentials $V$.

\begin{corollary}
    Let $(N,h)$ be a complete Riemannian manifold of dimension $n\ge 4$ and $K\subseteq N$ a closed set. Consider the open manifold $M:=N\setminus K$ and let $0\le V\in L^{2}_{loc}(M)$. If
    \begin{align}
        \mathfrak{M}_{2}(\widetilde{K})<+\infty \quad\quad \text{for every compact $\widetilde{K}\subseteq K$},
    \end{align}
    then $(M,g)$ is $(ESA_V)$.
\end{corollary}

\ \\

\noindent \textbf{AI Statement.} The authors acknowledge the use of artificial intelligence tools during the preparation and drafting of this manuscript. In particular, AI has been used to review the manuscript, identifying typos, minor errors, and inconsistencies, and suggesting simplifications to certain proofs, in particular Proposition \ref{Prop_L^p-Liou_lambda-sh implies L^p-Liou_SH} and Lemma \ref{equivalence capacities}. All final outputs were critically reviewed and approved by the authors, who maintain sole responsibility for the scientific content.

\ \\

\noindent \textbf{Acknowledgments.} A. Bisterzo and S. Lincastri are partially supported by the 2026 INdAM-GNAMPA project \emph{Asymptotic analysis of variational problems}, n. \texttt{CUP\_E53C25002010001}. A. Bisterzo is also supported by ``Centro di Ricerca Matematica Ennio De Giorgi''. G.Veronelli is also partially supported
by INdAM-GNAMPA. Moreover, this project was partially developped while G.V. being hosted at Simons semester in Geometric Analysis at Impan in June and July 2026: G.V. acknowledge that this work
was partially supported by the Simons Foundation grant (award no. SFI-MPS-T-Institutes-00010825) and from
State Treasury funds as part of a task commissioned by the Minister of Science and Higher Education under the 
project “Organization of the Simons Semesters at the Banach Center - New Energies in 2026-2028”(agreement no. MNiSW/2025/DAP/491).

\bibliographystyle{alpha}
\bibliography{References}
\end{document}